\documentclass[draft,reqno]{amsproc}
\usepackage{amsmath,amssymb}
\usepackage{euscript} 
\usepackage{mathrsfs} 
\usepackage{color,pifont}
\usepackage{tabularx}
\usepackage{makecell}

\makeatletter
\@namedef{subjclassname@2010}{%
\textup{2010} Mathematics Subject Classification}
\makeatother

\numberwithin{equation}{section}

\newtheorem{thm}{Theorem}[section]
\newtheorem{cor}[thm]{Corollary}
\newtheorem{lem}[thm]{Lemma}
\newtheorem{pro}[thm]{Proposition}

\newtheorem*{thm*}{Theorem}

\newtheorem*{opq*}{Problem}

\theoremstyle{remark}
\newtheorem{rem}[thm]{Remark}

\theoremstyle{definition}
\newtheorem{exa}[thm]{Example}

\newcommand*{\alphab}{\boldsymbol{\alpha}}
\newcommand*{\betab}{\boldsymbol{\beta}}
\newcommand*{\borel}[1]{{\mathfrak B}(#1)}

\newcommand*{\cbb}{\mathbb C}
\newcommand*{\D}{\mathrm{d}}

\newcommand*{\Ge}{\geqslant}
\newcommand*{\gammab}{\boldsymbol{\gamma}}
\newcommand*{\hh}{\mathcal H}

\newcommand*{\is}[2]{\langle#1,#2\rangle}

\newcommand*{\lambdab}{\boldsymbol{\lambda}}
\newcommand*{\Le}{\leqslant}

\newcommand*{\nbb}{\mathbb N}

\newcommand*{\ogr}[1]{\boldsymbol B(#1)}
\newcommand*{\ob}[1]{{\mathscr R}(#1)}
\newcommand*{\rbb}{\mathbb R}

\newcommand*{\supp}[1]{\mathrm{supp}(#1)}

\newcommand*{\wlam}{W_{\lambdab}}

\newcommand*{\zbb}{\mathbb Z}

\begin{document}
   \title[Conditionally positive definite unilateral weighted shifts] {Conditionally positive definite
unilateral weighted shifts}
   \author[Z.\ J.\ Jab{\l}o\'nski]{Zenon Jan
Jab{\l}o\'nski}
   \address{Instytut Matematyki,
Uniwersytet Jagiello\'nski, ul.\ \L ojasiewicza 6,
PL-30348 Kra\-k\'ow, Poland}
\email{Zenon.Jablonski@im.uj.edu.pl}
   \author[I.\ B.\ Jung]{Il Bong Jung}
   \address{Department of Mathematics, Kyungpook National University,
Da\-egu 41566, Korea}
   \email{ibjung@knu.ac.kr}
   \author[E. Y. Lee]{Eun Young Lee}
   \address{Department of Mathematics, Kyungpook National University,
Da\-egu 41566, Korea}
   \email{eunyounglee@knu.ac.kr}
   \author[J.\ Stochel]{Jan Stochel}
\address{Instytut Matematyki, Uniwersytet
Jagiello\'nski, ul.\ \L ojasiewicza 6, PL-30348
Kra\-k\'ow, Poland} \email{Jan.Stochel@im.uj.edu.pl}
   \thanks{The research of the second author was supported
by the National Research Foundation of Korea
(NRF) grant funded by the Korea Government
(MSIT) (2018R1A2B6003660). The third author was
supported by the Basic Science Research Program
through the National Research Foundation of
Korea (NRF) funded by the Ministry of Education
(NRF-2018R1A6A3A01012892).}



   \maketitle
   \begin{abstract}
In a recent paper \cite{Ja-Ju-St20}, Hilbert
space operators $T$ with the property that each
sequence of the form $\{\|T^n
h\|^2\}_{n=0}^{\infty}$ is conditionally
positive definite in a semigroup sense were
introduced. In the present paper, this line of
research is continued in depth in the case of
unilateral weighted shifts. The conditional
positive definiteness of unilateral weighted
shifts is characterized in terms of formal
moment sequences. The description of the
representing triplet, the main object
canonically associated with such operators, is
provided. The backward extension problem for
conditionally positive definite unilateral
weighted shifts is solved, revealing a new
feature that does not appear in the case of
other operator classes. Finally, the flatness
problem in this context is discussed, with an
emphases on unexpected differences from the
analogous problem for subnormal unilateral
weighted shifts.
   \end{abstract}
   \section{Introduction}
The celebrated Lambert's criterion for
subnormality states that a bounded linear
operator $T$ on a complex Hilbert space $\hh$ is
subnormal if and only if for every $h\in \hh$,
the sequence $\{\|T^n h\|^2\}_{n=0}^{\infty}$ is
positive definite as a function on the additive
semigroup of all nonnegative integers $\zbb_+$
(see \cite{lam}, see also \cite[Theorem~
7]{St-Sz89}). In the harmonic analysis on
$*$-semigroups presented in \cite{B-C-R},
related classes of functions appear that play
important role in various branches of
mathematics and probability theory. Among them,
the class of conditionally positive definite
functions prevails. In a recent paper
\cite{Ja-Ju-St20}, a new class of operators,
called conditionally positive definite (CPD for
brevity), was introduced and studied in depth.
The operator $T$ is said to be CPD if for every
$h\in \hh$, the sequence $\{\|T^n
h\|^2\}_{n=0}^{\infty}$ is conditionally
positive definite as a function on the semigroup
$\zbb_+$. The class of CPD operators contains in
particular subnormal operators \cite{Hal50,Con},
complete hypercontractions of order $2$
\cite{Cha-Sh}, $3$-isometries
\cite{Ag-St1,Ag-St2,Ag-St3} and many others.

In this paper we study CPD unilateral weighted
shifts with positive weights, the issue not
included in \cite{Ja-Ju-St20}. The study is
preceded by necessary preparations related to
selected properties of conditionally positive
definite sequences (see Section~ \ref{Sec.2}).
The main characterization of CPD unilateral
weighted shifts given in Theorem~\ref{cpdws} is
written in the spirit of the classical Berger
theorem on subnormal unilateral weighted shifts
\cite{g-w70,hal70}. Namely, a unilateral
weighted shift $\wlam$ with weights
$\lambdab=\{\lambda_n\}_{n=0}^{\infty}$ is
subnormal if and only if the formal moment
sequence $\hat\lambdab$ associated with $\wlam$
(see \eqref{mur-hupy}) is a Stieltjes moment
sequence. In other words, there is a one-to-one
correspondence between subnormal unilateral
weighted shifts and (non-degenerate) Stieltjes
moment sequences. In turn, CPD unilateral
weighted shifts are in a one-to-one
correspondence with normalized CPD sequences of
exponential growth whose terms are all positive
and whose representing measures are supported in
the closed half-line $[0,\infty)$ (see
Theorems~\ref{cpdws} and \ref{wkwcpdws}).

The essential difference between the above two
characterizations relies on the fact that
Stieltjes moment sequences have all terms
nonnegative, while CPD ones do not. For this
reason, we devoted one section for finding
necessary and sufficient conditions for
positivity of CPD sequences (see
Section~\ref{Sec.5}). It is worth emphasizing
that, unlike Stieltjes moment sequences which
are represented by single parameter (a positive
measure), CPD ones are represented by three
parameters coming from a L\'{e}vy-Khinchin type
formula. Note that the Berger theorem (see
Corollary~\ref{Ber-G-W}) can be derived from
Theorem~\ref{cpdws} by using some scaling
results from~\cite{Ja-Ju-St20}.

Another major difference between subnormal and CPD unilateral
weighted shifts is revealed when solving the backward
extension problem. Our solution given in Theorem~\ref{bireks}
(see also Theorem~\ref{nthbuk}) is the basis for constructing
a CPD unilateral weighted shift with weights $(\lambda_0,
\lambda_1, \ldots)$ for which the extended unilateral
weighted shift with weights $(t,\lambda_0, \lambda_1,
\ldots)$ is CPD for any positive real number $t$ (see
Example~\ref{muritru}). This differs from solving the
subnormal backward extension problem (see
\cite[Proposition~8]{Cur90}).

It was Stampfli who noticed that if two consecutive weights
of a subnormal unilateral weighted shift $\wlam$ are equal,
then all weights of $\wlam$ (except perhaps the first one)
are equal to each other (see \cite[Theorem~6]{Sta66}). As
shown in Theorems~\ref{4weights-1} and \ref{firwts}, this is
no longer true in the case of CPD unilateral weighted shifts.
Namely, the smallest number of consecutive equal weights is
four. If we allow the consecutive weights to be equal to one,
then the number decreases to $2$ (see
Theorems~\ref{2weights-1} and \ref{firetom}). The
corresponding counterexamples confirming the minimality of
these numbers are given in Examples~\ref{przyktwofor} and
\ref{gusv}.

The main object canonically associated with CPD operators is
the (operator) representing triplet. The description of the
representing triplet for CPD unilateral weighted shifts is
provided in Theorem~\ref{truplyt}. The question of when the
components of this triplet are compact operators is discussed
in detail in Propositions~\ref{bnocr}, \ref{Bdziadnc} and
\ref{cupuc}, and Example~\ref{oliun}.

We refer the reader to the classical treatise on
weighted shifts \cite{shi74}.
   \section{\label{Sec.2}Preliminaries} Let $\rbb$ and $\cbb$
stand for the fields of real and complex numbers
respectively and let $\rbb_+=\{x \in \rbb \colon
x \Ge 0\}$. Denote by $\zbb_+$ and $\nbb$ the
sets of nonnegative and positive integers
respectively. We write $\borel{X}$ for the
$\sigma$-algebra of Borel subsets of a
topological space $X$. Unless otherwise stated,
all scalar measures we consider in this paper
are assumed to be positive. Given a finite Borel
measure $\nu$ on $\rbb$, we denote by
$\supp{\nu}$ the closed support of $\nu$. We
write $\delta_t$ for the Borel probability
measure on $\rbb$ concentrated at the point
$t\in \rbb$. We say that a sequence
$\{\gamma_n\}_{n=0}^{\infty}$ of real numbers is
of {\em exponential growth} if $\limsup_{n\to
\infty}|\gamma_n|^{1/n} < \infty$, or
equivalently if and only if there exist
$\alpha,\theta \in \rbb_+$ such that
   \begin{align*}
|\gamma_n| \Le \alpha\, \theta^n, \quad n\in
\zbb_+.
   \end{align*}
The discrete differentiation transformation
$\triangle\colon \rbb^{\zbb_+} \to \rbb^{\zbb_+}$
is given by
   \begin{align*}
(\triangle \gammab)_n = \gamma_{n+1} - \gamma_n,
\quad n\in \zbb_+, \, \gammab =
\{\gamma_n\}_{n=0}^{\infty} \in \rbb^{\zbb_+}.
   \end{align*}
We denote by $\triangle^k$ the $k$th composition
power of $\triangle$. Note that if $\gammab$ is
of exponential growth, so is $\triangle \gammab$.

Given $n\in \zbb_+$, we define the polynomial $Q_n$ in
real variable $x$ by
   \begin{align} \label{klaud}
Q_n(x) =
   \begin{cases}
0 & \text{if } n=0,1,
   \\
\sum_{j=0}^{n-2} (n -j -1) x^j & \text{if } n\Ge 2,
   \end{cases}
\quad x \in \rbb.
   \end{align}
It is immediate from definition that
   \begin{gather} \label{immed-pos}
Q_n(x) \Ge 1, \quad x \in \rbb_+, \, n\Ge 2.
   \end{gather}
Below, $\triangle^j Q_{(\cdot)}(x)$ denotes the
action of the transformation $\triangle^j$ on
the sequence $\{Q_{n}(x)\}_{n=0}^{\infty}$ for
$x\in \rbb$ and $j\in \nbb$. The polynomials
$\{Q_n\}_{n=0}^{\infty}$ have the following
properties that will be used in the subsequent
parts of the paper (see
\cite[Lemma~2.2.1]{Ja-Ju-St20}):
   \allowdisplaybreaks
   \begin{align}  \label{rnx-1}
Q_n(x) & = \frac{x^n-1 - n (x-1)}{(x-1)^2}, \quad n \in
\zbb_+, \, x\in \rbb\setminus \{1\},
   \\ \label{rnx-0}
Q_{n+1}(x) & = x Q_n(x) + n, \quad n \in \zbb_+, \, x\in
\rbb,
   \\ \label{monot-1}
\frac{Q_n(x)}{n} & \Le \frac{Q_{n+1}(x)}{n+1}, \quad n\Ge
1, \, x \in [0,1],
   \\ \label{monot-2}
\lim_{n\to \infty} \frac{Q_n(x)}{n} & =
\frac{1}{1-x}, \quad x\in [0,1),
   \\ \label{del1}
(\triangle Q_{(\cdot)}(x))_n & =
   \begin{cases}
0 & \text{if } n=0, \, x\in \rbb,
   \\
\sum_{j=0}^{n-1} x^j & \text{if } n \in \nbb, \,
x\in \rbb,
   \end{cases}
   \\ \label{del2}
(\triangle^2 Q_{(\cdot)}(x))_n & = x^n, \quad
n\in \zbb_+, \, x\in \rbb.
   \end{align}
We need also the following two additional properties of
polynomials $Q_n$.
   \begin{lem} \label{huhu-zima}
The following two assertions are valid{\em :}
   \begin{gather} \label{fur-1}
\frac{Q_n(x)}{n^2} \Le 1, \quad x \in [0,1], \,
n\Ge 1,
   \\ \label{fur-2}
\lim_{n\to \infty} \int_{[0,1)} \frac{Q_n}{n^2}
\D \nu = 0 \quad \text{if $\nu$ is a finite
Borel measure on $[0,1)$}.
   \end{gather}
   \end{lem}
   \begin{proof}
It follows from \eqref{klaud} that
   \begin{align*}
Q_n(x) \Le 1 + 2 + \ldots + (n-1) =
\frac{n(n-1)}{2} \Le n^2, \quad x \in [0,1], \,
n\Ge 2,
   \end{align*}
which implies \eqref{fur-1}. By \eqref{monot-2},
$\lim_{n\to\infty} \frac{Q_n(x)}{n^2} = 0$ for
all $x\in [0,1)$. Combined with \eqref{fur-1}
and Lebesgue's dominated convergence theorem,
this yields \eqref{fur-2}.
   \end{proof}

We now recall some basic facts from harmonic
analysis on semigroups (see \cite{B-C-R}; see
also \cite{Ja-Ju-St20}). Let $\gammab =
\{\gamma_n\}_{n=0}^{\infty}$ be a sequence of
real numbers. We say that $\gammab$ is {\em
positive definite} ({\em PD} for brevity) if
   \begin{align} \label{wicher}
\sum_{i,j=0}^k \gamma_{i+j} \lambda_i \bar\lambda_j \Ge 0,
   \end{align}
for all finite sequences $\lambda_0, \ldots,
\lambda_k \in \cbb$. If the inequality
\eqref{wicher} holds for all finite sequences
$\lambda_0, \ldots, \lambda_k \in \cbb$ such
that $\sum_{j=0}^k \lambda_j=0$, then we call
$\gammab$ {\em conditionally positive definite}
({\em CPD} for brevity). Obviously, PD sequences
are CPD but not conversely. We say that
$\gammab$ is a {\em Stieltjes moment sequence}
if there exists a finite Borel measure $\mu$ on
$\rbb_+$, called a {\em representing measure} of
$\gammab$, such that
   \begin{align*}
\gamma_n = \int_{\rbb_+} x^n \D \mu(x), \quad n
\in \zbb_+.
   \end{align*}
By the Stieltjes theorem (see
\cite[Theorem~6.2.5]{B-C-R}), $\gammab$ is a
Stieltjes moment sequence if and only if the
sequences $\{\gamma_n\}_{n=0}^{\infty}$ and
$\{\gamma_{n+1}\}_{n=0}^{\infty}$ are PD. A
Stieltjes moment sequence $\gammab$ is called
{\em non-degenerate} if $\gamma_n > 0$ for all
$n\in \zbb_+$. The following observation is a
direct consequence of the definition.
   \begin{align} \label{foot-1}
   \begin{minipage}{70ex}
{\em A Stieltjes moment sequence $\gammab$ is
non-degenerate if and only if $\gamma_n
> 0$ for some $n\Ge 1$.}
   \end{minipage}
   \end{align}
We also need the following property of
non-degenerate Stieltjes moment sequences.
   \begin{lem} \label{ajjaj}
If $k\in \zbb_+$ and
$\{\gamma_n\}_{n=0}^{\infty}$ is a
non-degenerate Stieltjes moment sequence with a
representing measure $\mu$, then the following
conditions are~equivalent{\em :}
   \begin{enumerate}
   \item[(i)] $\gamma_{k+1}^2 = \gamma_k
\gamma_{k+2}$,
   \item[(ii)] there exists $\zeta \in (0,\infty)$
such that $\supp{\mu} = \{\zeta\}$ if $k=0$, or
$\supp{\mu} \subseteq \{0, \zeta\}$ if $k\Ge 1$.
   \end{enumerate}
   \end{lem}
   \begin{proof}
Adapt the proof of \cite[Lemma~3.4]{Ja-Ju-St11}.
   \end{proof}
In this paper we deal mainly with CPD sequences
of exponential growth. They can be described as
follows.
   \begin{thm}[{\cite[Theorem~2.2.5]{Ja-Ju-St20}}]
\label{cpd-expon}
   Let $\gammab=\{\gamma_n\}_{n=0}^{\infty}$ be a
sequence of real numbers. Then the following conditions
are equivalent{\em :}
   \begin{enumerate}
   \item[(i)] $\gammab$ is a CPD sequence of exponential growth,
   \item[(ii)] there exist  $b\in \rbb$,
$c\in \rbb_+$ and a compactly supported finite Borel
measure $\nu$ on $\rbb$ such that $\nu(\{1\})=0$ and
   \begin{align} \label{cdr4}
\gamma_n = \gamma_0 + bn + c n^2 + \int_{\rbb} Q_n(x)
\D\nu(x), \quad n\in \zbb_+.
   \end{align}
   \end{enumerate}
Moreover, if {\em (ii)} holds, then the triplet
$(b,c,\nu)$ is unique and
   \allowdisplaybreaks
   \begin{gather*}
\supp{\nu} \subseteq \Big[-\limsup_{n\to
\infty}|\gamma_n|^{1/n}, \limsup_{n\to
\infty}|\gamma_n|^{1/n}\Big].
   \end{gather*}
   \end{thm}
Call $(b,c,\nu)$ appearing in
Theorem~\ref{cpd-expon}(ii) the {\em
representing triplet} of $\gammab$. By
\eqref{klaud} and \eqref{cdr4} we see that if
$(b,c,\nu)$ is the representing triplet of
$\gammab$, then
   \begin{align} \label{euju}
b = \gamma_1- \gamma_0 - c.
   \end{align}
For $k\in \zbb_+$, we define the shifted sequence
$\gammab^{(k)} =
\{\gamma^{(k)}_n\}_{n=0}^{\infty}$ by
   \begin{align} \label{bekon}
\gamma^{(k)}_n =\gamma_{k+n}, \quad n\in \zbb_+,
\, k \in \zbb_+.
   \end{align}
Observe that
   \begin{align} \label{fateu}
\triangle^p (\gammab^{(k)}) = (\triangle^p
\gammab)^{(k)}, \quad p, k \in \zbb_+.
   \end{align}
It follows straightforwardly from definition
that if $\gammab$ is CPD and $k$ is an even
positive number, then $\gammab^{(k)}$ is CPD.
This is not true for any positive odd number $k$
(e.g., $\gamma_n=(-1)^n$ for $n\in \zbb_+$). The
description of the representing triplet of
$\gammab^{(k)}$ is provided in
Lemma~\ref{szuf-eru} below. We give a direct
proof of this lemma. Another way of proving it
is to use induction and the recurrence formula
\eqref{rnx-0}; we leave the details to the
interested reader.
   \begin{lem} \label{szuf-eru}
Let $\gammab=\{\gamma_n\}_{n=0}^{\infty}$ be a
CPD sequence of exponential growth and let
$(b,c,\nu)$ be the representing triplet of
$\gammab$. Suppose that $k\in \nbb$. Then
$\gammab^{(k)}$ is of exponential growth and the
following statements hold{\em :}
   \begin{enumerate}
   \item[(i)]  if $k$ is even,
then $\gammab^{(k)}$ is CPD,
   \item[(ii)] if
$k$ is odd, then $\gammab^{(k)}$ is CPD if and
only if $\supp{\nu} \subseteq \rbb_+$,
   \item[(iii)] if $\gammab^{(k)}$ is CPD, then
   \begin{align}  \label{haha1}
b_k & = \gamma_{k+1} - \gamma_{k} -c = b + 2kc +
\sum_{j=0}^{k-1} \int_{\rbb} x^j \D \nu(x),
   \\ \label{haha2}
c_k & = c,
   \\ \label{haha3}
\nu_k(\varDelta) & = \int_{\varDelta} x^k \D
\nu(x), \quad \varDelta \in \borel{\rbb},
   \end{align}
where $(b_k,c_k,\nu_k)$ is the representing
triplet $\gammab^{(k)}$.
   \end{enumerate}
   \end{lem}
   \begin{proof}
It follows from \eqref{del2} and \eqref{cdr4}
that
   \begin{align} \label{sniug}
(\triangle^2 \gammab)_n = 2 c + \int_{\rbb} x^n
\D\nu(x), \quad n \in \zbb_+.
   \end{align}
This implies that
   \allowdisplaybreaks
   \begin{align} \notag
(\triangle^2 \gammab^{(k)})_n
\overset{\eqref{fateu}}= (\triangle^2
\gammab)_{k+n} & \overset{\eqref{sniug}}= 2 c +
\int_{\rbb} x^n \D\tilde\nu_k(x)
   \\ \label{aru-fr}
& \hspace{1.5ex}= \int_{\rbb} x^n (\tilde\nu_k +
2 c \delta_1)(\D x), \quad n \in \zbb_+,
   \end{align}
where $\tilde\nu_k$ is the signed Borel measure
on $\rbb$ given by
   \begin{align} \label{szes-dwu}
\tilde \nu_k(\varDelta) = \int_{\varDelta} x^k
\D\nu(x), \quad \varDelta \in \borel{\rbb}.
   \end{align}

If $k$ is odd and $\supp{\nu} \subseteq \rbb_+$,
then the measures $\tilde\nu_k$ and $\tilde\nu_k
+ 2 c \delta_1$ are positive. Combined with
\cite[Proposition~2.2.9]{Ja-Ju-St20} and the
fact that $\gammab^{(k)}$ is of exponential
growth, we infer from \eqref{aru-fr} that the
sequence $\gammab^{(k)}$ is CPD, which proves
the ``if'' part of (ii).

Suppose now that $\gammab^{(k)}$ is CPD.
Applying \eqref{sniug} to $\gammab^{(k)}$, we
obtain
   \begin{align} \label{aru-fu}
(\triangle^2 \gammab^{(k)})_n = 2 c_k +
\int_{\rbb} x^n \D\nu_k(x) = \int_{\rbb} x^n
(\nu_k + 2 c_k \delta_1)(\D x), \quad n \in
\zbb_+.
   \end{align}
Since the measures $\tilde\nu_k + 2 c \delta_1$
and $\nu_k + 2 c_k \delta_1$ are compactly
supported (the first of them may not be
positive), we deduce from \eqref{aru-fr},
\eqref{aru-fu} and \cite[Lemma~4.1]{C-J-J-S21}
that
   \begin{align*}
\tilde\nu_k + 2 c \delta_1 = \nu_k + 2 c_k
\delta_1.
   \end{align*}
Since $\tilde\nu_k(\{1\})=\nu_k(\{1\})=0$, this
implies that $c_k=c$ and $\tilde\nu_k=\nu_k$,
which yields \eqref{haha2} and \eqref{haha3}.
Applying \eqref{euju} to $\gammab^{(k)}$ and then
using \eqref{cdr4} and \eqref{del1}, we obtain
\eqref{haha1}. This proves (iii). If moreover $k$
is odd, then using \eqref{szes-dwu} and the fact
that $\tilde\nu_k=\nu_k$ is a positive measure,
we deduce that $\supp{\nu} \subseteq \rbb_+$.
This proves the ``only if'' part of (ii).

As mentioned just before Lemma~\ref{szuf-eru},
the statement (i) holds, and so the proof is
complete.
   \end{proof}
Given a (complex) Hilbert space $\hh$, we write
$\ogr{\hh}$ for the $C^*$-algebra of bounded
linear operators on $\hh$. The range of $T\in
\ogr{\hh}$ is denoted by $\ob{T}$. By a {\em
semispectral measure} on $\rbb_+$ we mean a map
$F\colon \borel{\rbb_+}\to \ogr{\hh}$ defined on
the $\sigma$-algebra $\borel{\rbb_+}$ of Borel
subsets of $\rbb_+$ for which
$\is{F(\cdot)h}{h}$ is a measure for every $h\in
\hh$. The closed support of $F$ is denoted by
$\supp{F}$. Following \cite{Ja-Ju-St20}, we say
that an operator $T\in \ogr{\hh}$ is {\em
conditionally positive definite} ({\em CPD} for
brevity) if the sequence $\{\|T^n
h\|^2\}_{n=0}^{\infty}$ is CPD for every $h\in
\hh$. By Lambert's criterion for subnormality,
subnormal operators are CPD. The CPD operators
can be characterized as follows.
   \begin{thm}[{\cite[Theorem~3.1.1]{Ja-Ju-St20}}] \label{cpdops}
Let $T\in \ogr{\hh}$. Then the following statements are
equivalent{\em :}
   \begin{enumerate}
   \item[(i)] $T$ is CPD,
   \item[(ii)] there exist operators $B,C\in
\ogr{\hh}$ and a compactly supported semispectral measure
$F\colon \borel{\rbb_+} \to \ogr{\hh}$ such that $B=B^*$,
$C\Ge 0$, $F(\{1\})=0$ and
   \begin{align*}
T^{*n}T^n = I + n B + n^2 C + \int_{\rbb_+} Q_n(x) F(\D
x), \quad n\in \zbb_+.
   \end{align*}
   \end{enumerate}
Moreover, if {\em (ii)} holds, then the triplet $(B,C,F)$
is unique and
   \begin{gather*}
\supp{F} \subseteq [0,r(T)^2].
   \end{gather*}
Furthermore, $(\is{Bh}h, \is{Ch}h,
\is{F(\cdot)h}h)$ is the representing triplet of
the CPD sequence $\{\|T^n
h\|^2\}_{n=0}^{\infty}$ for every $h\in \hh$.
   \end{thm}
Following \cite{Ja-Ju-St20}, we call the triplet
$(B,C,F)$ appearing in Theorem~\ref{cpdops}(ii)
the {\em representing triplet} of $T$.
   \section{A characterization of CPD unilateral weighted shifts}
In this paper we consider unilateral weighted shifts with
positive weights (see \cite{shi74} for the fundamental
properties of weighted shifts). As usual $\ell^2$ stands for
the Hilbert space of all square summable complex sequences
$\{\alpha_n\}_{n=0}^{\infty}$. Given a bounded sequence
$\lambdab=\{\lambda_n\}_{n=0}^{\infty}$ of positive real
numbers, there is a unique $\wlam\in \ogr{\ell^2}$ such that
   \begin{align*}
\wlam e_n = \lambda_n e_{n+1}, \quad n \in
\zbb_+,
   \end{align*}
where $\{e_n\}_{n=0}^{\infty}$ stands for the
standard orthonormal basis of $\ell^2$. We call
$\wlam$ a {\em unilateral weighted shift} with
weights $\lambdab$. If $\lambda_n=1$ for all $n
\in \zbb_+$, then we call $\wlam$ the {\em
unilateral shift}. In this paper we consider
only bounded unilateral weighted shifts with
positive real weights.

Suppose that $\wlam$ is a unilateral weighted
shift with weights
$\lambdab=\{\lambda_n\}_{n=0}^{\infty}$. The
sequence
$\hat\lambdab=\{\hat\lambda_n\}_{n=0}^{\infty}$
associated to $\wlam$ is defined by\footnote{It
is sometimes called the moment sequence of
$\wlam$. Obviously, this terminology is not
appropriate in our paper.}
   \begin{align} \label{mur-hupy}
\hat\lambda_n =
   \begin{cases}
1 & \text{if } n=0,
   \\
\lambda_0^2 \cdots \lambda_{n-1}^2 & \text{if }
n \Ge 1,
   \end{cases}
\quad n\in \zbb_+.
   \end{align}
It is easily seen that $\hat\lambdab$ is of exponential
growth and the following identity holds:
   \begin{align} \label{bas-pow}
\hat\lambda_n = \|\wlam^n e_0\|^2, \quad n\in \zbb_+.
   \end{align}
We also have
   \begin{align} \label{Eun-1}
(\widehat{\theta \lambdab})_n = \theta^{2n}
\hat{\lambda}_n, \quad n \in \zbb_+, \, \theta
\in (0,\infty),
   \end{align}
where $\theta \lambdab := \{\theta
\lambda_n\}_{n=0}^{\infty}$. The sequence $\lambdab$
of weights of $\wlam$ can be recaptured immediately
from the sequence $\hat\lambdab$ via the following
formula
   \begin{align} \label{self-map}
\lambda_n = \sqrt{\frac{\hat\lambda_{n+1}}{\hat\lambda_{n}}}, \quad n
\in \zbb_+.
   \end{align}
Note that the transformation $\lambdab \mapsto
\hat \lambdab$ given by \eqref{mur-hupy} is a
one-to-one correspondence between the set of
sequences of positive real numbers and the set
of sequences of positive real numbers with the
first term equal to $1$. The celebrated Berger
theorem states that a unilateral weighted shift
$\wlam$ is subnormal if and only if the
corresponding sequence $\hat\lambdab$ is a
Stieltjes moment sequence (see
\cite{g-w70,hal70}). In other words, the
transformation $\lambdab \mapsto \hat \lambdab$
is a one-to-one correspondence between the set
of weights of subnormal unilateral weighted
shifts and the set of non-degenerate Stieltjes
moment sequences with compactly supported
representing probability measures. We will show
that a similar effect occurs in the case of CPD
unilateral weighted shifts. Namely, the
transformation $\lambdab \mapsto \hat \lambdab$
gives a one-to-one correspondence between the
set of weights of CPD unilateral weighted shifts
and the set of normalized CPD sequences of
exponential growth with positive terms and with
representing measures supported in $\rbb_+$ (see
Theorem~\ref{cpdws}, see also
Theorem~\ref{wkwcpdws}). The main difference
between these two characterizations is that the
terms of a non-degenerate Stieltjes moment
sequence are automatically positive (cf.\
\eqref{foot-1}), while some terms of a CPD
sequence may not be. What is more, it is not
easy to determine CPD sequences with positive
weights. One of our goals in this paper is to
describe CPD sequences of exponential growth
with positive terms.
   \begin{thm} \label{cpdws}
Let $\lambdab=\{\lambda_n\}_{n=0}^{\infty}$ be a bounded sequence of
positive real numbers. Then the following conditions are
equivalent{\em :}
   \begin{enumerate}
   \item[(i)] the operator $\wlam$ is CPD,
   \item[(ii)] there exist $b\in \rbb$, $c\in \rbb_+$
and a compactly supported finite Borel measure $\nu$
on $\rbb_+$ such that $\nu(\{1\})=0$ and
   \begin{align} \label{wnezero}
\hat\lambda_n = 1 + bn + c n^2 + \int_{\rbb_+} Q_n(x) \D\nu(x), \quad
n\in \zbb_+.
   \end{align}
   \end{enumerate}
Moreover, the triplet $(b,c,\nu)$ appearing in
{\em (ii)} is unique and
   \begin{align*}
\supp{\nu} \subseteq \big[0, \limsup_{n\to \infty}
\hat\lambda_n^{1/n}\big].
   \end{align*}
   \end{thm}
We call the triplet $(b,c,\nu)$ appearing in
Theorem~\ref{cpdws}(ii) the {\em scalar
representing triplet} of $\wlam$.
   \begin{proof}[Proof of Theorem~\ref{cpdws}]
We begin by observing that since $\lambdab$ is bounded,
the sequence $\hat\lambdab$ is of exponential growth.

(i)$\Rightarrow$(ii) Apply \eqref{bas-pow} and
Theorem~\ref{cpdops}.

(ii)$\Rightarrow$(i) It follows from
\eqref{bas-pow} that
   \begin{align} \label{wnuk}
\|\wlam^n h\|^2 = \sum_{j=0}^{\infty} |\is{h}{e_j}|^2
\|\wlam^n e_j\|^2 = \sum_{j=0}^{\infty} |\is{h}{e_j}|^2
\frac{\hat\lambda_{n+j}}{\hat\lambda_{j}}, \quad n \in
\zbb_+, \, h \in \ell^2.
   \end{align}
According to Theorem~\ref{cpd-expon} and
Lemma~\ref{szuf-eru}, the sequence
$\{\hat\lambda_{n+j}\}_{n=0}^{\infty}$ is CPD
for all $j\in \zbb_+$. It is a matter of routine
to deduce from \eqref{wnuk} that the sequence
$\{\|\wlam^n h\|^2\}_{n=0}^{\infty}$ is CPD for
all $h\in \ell^2$, which means that $\wlam$ is
CPD.

The ``moreover'' part is a direct consequence of
Theorem~\ref{cpd-expon}.
   \end{proof}
It is of some interest that the famous Berger criterion
\cite{g-w70,hal70} for subnormality of unilateral weighted
shifts can be deduced easily from Theorem~\ref{cpdws} via the
theory of CPD operators.
   \begin{cor}[Berger theorem] \label{Ber-G-W}
Let $\lambdab=\{\lambda_n\}_{n=0}^{\infty}$ be a bounded
sequence of positive real numbers. Then the following
conditions are equivalent{\em :}
   \begin{enumerate}
   \item[(i)] the operator $\wlam$ is subnormal,
   \item[(ii)] there exists a Borel
measure $\mu$ on $\rbb_+$ such that
   \begin{align}  \label{Stieq}
\hat \lambda_n = \int_{\rbb_+} x^n \D \mu(x), \quad n \in
\zbb_+.
   \end{align}
   \end{enumerate}
Moreover, the measure $\mu$ in {\em (ii)} is
unique, compactly supported, and finite.
   \end{cor}
   \begin{proof}
(i)$\Rightarrow$(ii) Fix $\theta \in (0,\infty)$ such
that
   \begin{align} \label{Eun-2}
\theta^2 \limsup_{n\to \infty} (\hat
\lambda_n)^{1/n} < 1.
   \end{align}
Since $W_{t \lambdab} = t \wlam$ is subnormal
and consequently CPD for all $t\in (0,\infty)$,
we infer from \eqref{bas-pow} applied to
$W_{\theta \lambdab}$ that $\widehat{\theta
\lambdab}$ is CPD. By \eqref{Eun-1},
\eqref{Eun-2} and
\cite[Theorem~2.2.13]{Ja-Ju-St20},
$\hat\lambdab$ is PD. Applying the above to
$W_{\alphab}$ with $\alphab =
\{\lambda_{n+1}\}_{n=0}^{\infty}$ ($W_{\alphab}$
is unitarily equivalent to
$\wlam|_{\overline{\ob{\wlam}}}$) and noting
that $\hat\lambda_{n+1} = \lambda_0^2 \,
\hat\alpha_n$ for all $n \in \zbb_+$, we see
that $\{\hat\lambda_{n+1}\}_{n=0}^{\infty}$ is
PD. Using \cite[Theorem~6.2.5]{B-C-R} yields
(ii).

(ii)$\Rightarrow$(i) Fix $\theta \in (0,\infty)$
such that $\|\theta\wlam\| < 1$. By assumption
and \eqref{Eun-1}, the sequence $\widehat{\theta
\lambdab}$ is a Stieltjes moment sequence and by
\cite[Theorem~2.2.12]{Ja-Ju-St20} it is CPD with
the representing triplet
$(b_{\theta},0,\nu_{\theta})$ such that
$\supp{\nu_{\theta}} \subseteq \rbb_+$. Using
Theorem~\ref{cpdws}, we deduce that $W_{\theta
\lambdab}$ is a CPD operator such that
$\|W_{\theta \lambdab}\|= \|\theta \wlam\| < 1$.
By \cite[Theorem~4.1]{Sto91}, $\theta \wlam$ and
consequently $\wlam$ are subnormal.

The ``moreover'' part is a direct consequence of
the fact that $\hat\lambdab$ is of exponential
growth (see e.g., \cite[(1.4)]{B-J-J-S18}).
   \end{proof}
   \begin{rem}
In the case when $\wlam$ is subnormal, the
relationship between the measure $\mu$ appearing
in Corollary~\ref{Ber-G-W}(ii), which is called
the {\em Berger measure} of $\wlam$, and the
triplet $(b,c,\nu)$ in Theorem~\ref{cpdws}(ii)
is as follows:
   \begin{gather*}
b = \int_{\rbb_+} (x-1) \D \mu(x), \quad c=0,
   \\ \label{zeg2}
\nu(\varDelta) = \int_{\varDelta} (x-1)^2 \D \mu(x),
\quad \varDelta \in \borel{\rbb_+},
   \\
\mu(\varDelta) =\int_{\varDelta} \frac{1}{(x-1)^2} \D
\nu(x) + \Big(\gamma_0 - \int_{\rbb_+} \frac{1}{(x-1)^2}
\D \nu(x)\Big) \delta_1(\varDelta), \quad \varDelta \in
\borel{\rbb_+}.
   \end{gather*}
We refer the reader to
\cite[Theorem~2.2.12]{Ja-Ju-St20} for more
details.
   \hfill $\diamondsuit$
   \end{rem}
   \section{Modelling CPD unilateral weighted shifts}
It follows from Theorem~\ref{cpdws} that if a
unilateral weighted shift $\wlam$ is CPD, then
the sequence $\gammab=\hat\lambdab$ has positive
real terms and takes the form \eqref{wnezero}.
Now we show the converse, namely if $\gammab$ is
a sequence of positive real numbers which is of
the form \eqref{wnezero}, then $\gammab =
\hat\lambdab$, where $\lambdab$ is a bounded
sequence of positive real numbers and so the
unilateral weighted shift $\wlam$ is CPD.
   \begin{thm} \label{wkwcpdws} Let
$\gammab=\{\gamma_{n}\}_{n=0}^{\infty}$ be a sequence of
real numbers such that
   \begin{gather} \label{gl-odd1}
\gamma_{n} > 0, \quad n\in \zbb_+,
   \\ \label{gl-odd2}
\gamma_{n} = 1 + bn + c n^2 + \int_{\rbb_+} Q_n \D\nu,
\quad n\in \zbb_+,
   \end{gather}
where $b\in \rbb$, $c\in \rbb_+$ and $\nu$ is a
compactly supported finite Borel measure on $\rbb_+$
such that $\nu(\{1\})=0$. Then the sequence
$\lambdab=\{\lambda_n\}_{n=0}^{\infty}$ defined by
   \begin{align} \label{wagua}
\lambda_n=\sqrt{\frac{\gamma_{n+1}}{\gamma_{n}}}, \quad
n\in \zbb_+,
   \end{align}
is bounded and the unilateral weighted shift
$\wlam$ is CPD. Moreover,
$\hat\lambdab=\gammab$.
   \end{thm}
   \begin{proof}
In view of Theorem~\ref{cpdws} and the obvious
identity $\hat\lambdab = \gammab$, it suffices
to show that the sequence $\lambdab$ is bounded,
or equivalently that
   \begin{align}  \label{lim-sup-r}
\limsup_{n\to \infty} \frac{\gamma_{n+1}}{\gamma_{n}} <
\infty.
   \end{align}
Before we proceed with the proof, we will make the
necessary preparations. Set $\vartheta=\sup\supp{\nu}$
with the standard convention that
$\sup{\emptyset}=-\infty$. First, we define the
polynomials $p$ and $q$ in $n$ by
   \begin{align*}
p(n) & = 1 + b(n+1) + n \nu(\rbb_+) + c (n+1)^2,
   \\
q(n) & = 1 + bn + c n^2.
   \end{align*}
Secondly, we observe that according to \eqref{immed-pos},
\eqref{rnx-0}, \eqref{gl-odd1} and \eqref{gl-odd2}, we
have
   \begin{align} \notag
\frac{\gamma_{n+1}}{\gamma_{n}} & = \frac{p(n) +
\int_{\rbb_+} x \,Q_n(x) \D\nu(x)}{q(n) + \int_{\rbb_+}
Q_n(x) \D\nu(x)}
   \\  \label{du-du-1}
& \Le \frac{p(n) + \max\{0,\vartheta\} \int_{\rbb_+}
Q_n \D\nu}{q(n) + \int_{\rbb_+} Q_n \D\nu}, \quad n\in
\zbb_+.
   \end{align}
Thirdly, using \eqref{immed-pos}, \eqref{monot-1} and
\eqref{monot-2} and applying Lebesgue's monotone
convergence theorem, we deduce that
   \begin{align} \label{Eun-J}
\lim_{n\to \infty} \int_{\rbb_+} \frac{Q_n}{n} \D\nu =
\int_{\rbb_+} \frac{1}{1-x} \D\nu(x) \quad
\text{whenever $\vartheta \Le 1$.}
   \end{align}

Now we go to the main part of the proof. We will split it
into several cases.

{\sf Case 1.} $\vartheta > 1$.

Take any $\theta \in (1,\vartheta)$. Since
$\vartheta \in \supp{\nu}$, we see that
$\nu([\theta,\vartheta]) > 0$. Using
\eqref{immed-pos}, we get
   \begin{align} \notag
\int_{\rbb_+} Q_n \D \nu & \Ge
\int_{[\theta,\vartheta]} Q_n \D \nu
   \\ \notag
& = \int_{[\theta,\vartheta]} (x^{n-2} +
2x^{n-3} + \ldots + (n-1) x^0\big) \D \nu(x)
   \\ \notag
& \Ge \int_{[\theta,\vartheta]} x^{n-2} \D
\nu(x)
   \\ \label{cul-dra}
& \Ge \theta^{n-2} \nu([\theta,\vartheta]),
\quad n \Ge 2.
   \end{align}
This in turn yields
   \begin{align} \label{str-2-zer}
\lim_{n \to \infty} \frac{n^2}{\int_{\rbb_+} Q_n \D\nu}
=0.
   \end{align}
Dividing the numerator and denominator of the last
fraction in \eqref{du-du-1} by $\int_{\rbb_+} Q_n
\D\nu$ and using \eqref{str-2-zer}, we deduce that
\eqref{lim-sup-r} holds.

{\sf Case 2.} $c>0$.

One can show that there exists $\alpha \in
(1,\infty)$ such that
   \begin{align} \label{du-du-2}
0 < p(n) \Le \alpha q(n) \textrm{ for $n$ large enough.}
   \end{align}
Hence by \eqref{du-du-1} we see that for $n$ large
enough,
   \begin{align*}
\frac{\gamma_{n+1}}{\gamma_{n}} & \Le \frac{p(n) +
\max\{0,\vartheta\} \int_{\rbb_+} Q_n \D\nu}{q(n) +
\int_{\rbb_+} Q_n \D\nu} \overset{\eqref{du-du-2}} \Le
\max \{\alpha, \vartheta\},
   \end{align*}
which yields \eqref{lim-sup-r}.

{\sf Case 3.} $\vartheta \Le 1$, $c=0$ and
$\int_{\rbb_+} \frac{1}{1-x} \D\nu(x) = \infty$.

Since $\int_{\rbb_+} Q_n \D\nu
>0$ for all integers $n\Ge 2$ (see \eqref{immed-pos}),
we get
   \begin{align*}
\frac{\gamma_{n+1}}{\gamma_{n}} &
\overset{\eqref{du-du-1}} \Le \frac{p(n)
+\int_{\rbb_+} Q_n \D\nu}{q(n) + \int_{\rbb_+} Q_n
\D\nu} = \frac{\frac{p(n)/n}{\int_{\rbb_+}
\frac{Q_n}{n} \D\nu} + 1}{\frac{q(n)/n}{\int_{\rbb_+}
\frac{Q_n}{n} \D\nu} + 1} , \quad n \Ge 2.
   \end{align*}
This together with \eqref{Eun-J} implies that
$\limsup_{n\to \infty} \frac{\gamma_{n+1}}{\gamma_{n}}
\Le 1$.

{\sf Case 4.} $\vartheta \Le 1$, $c=0$, $\int_{\rbb_+}
\frac{1}{1-x} \D\nu(x) < \infty$ and $b+\int_{\rbb_+}
\frac{1}{1-x} \D\nu(x)\neq 0$.

We can argue as follows:
   \begin{align} \notag
\lim_{n\to \infty}
\frac{\gamma_{n+1}}{\gamma_{n}} & = \lim_{n\to
\infty} \frac{\frac{1}{n} + \frac{n+1}{n}b +
\frac{n+1}{n}\int_{\rbb_+} \frac{Q_{n+1}}{n+1}
\D\nu}{\frac{1}{n} + b + \int_{\rbb_+}
\frac{Q_{n}}{n} \D\nu}
   \\ \label{Hu-hy}
& \hspace{-1ex}\overset{\eqref{Eun-J}}=
\frac{b+\int_{\rbb_+} \frac{1}{1-x}
\D\nu(x)}{b+\int_{\rbb_+} \frac{1}{1-x}
\D\nu(x)}=1.
   \end{align}

{\sf Case 5.} $\vartheta \Le 1$, $c=0$,
$\int_{\rbb_+} \frac{1}{1-x} \D\nu(x) < \infty$
and $b+\int_{\rbb_+} \frac{1}{1-x} \D\nu(x) = 0$.

First, observe that
   \begin{align} \notag
\gamma_n & \overset{\eqref{gl-odd2}}=1 + n \Big(b +
\int_{[0,1)} \frac{Q_n}{n} \D \nu\Big)
   \\  \notag
& \overset{\eqref{rnx-1}} = 1 + n \bigg(\Big(b +
\int_{[0,1)} \frac{1}{1-x} \D\nu(x)\Big) - \int_{[0,1)}
\frac{1-x^n}{n(1-x)^2} \D \nu(x)\bigg)
   \\  \label{bi-pas}
& \hspace{1ex} = 1 - \int_{[0,1)}
\frac{1-x^n}{(1-x)^2} \D \nu(x), \quad n \in \nbb.
   \end{align}
This implies that the sequence $\{\gamma_n\}_{n=1}^{\infty}$
is monotonically decreasing and so by \eqref{gl-odd1},
$\frac{\gamma_{n+1}}{\gamma_n} \Le 1$ for all $n\Ge 1$, which
completes the proof.
   \end{proof}
Note that Theorem~\ref{wkwcpdws} provides a
method of obtaining CPD unilateral weighted
shifts from positive CPD sequences of
exponential growth. The question of positivity
of CPD sequences will be discussed in the next
section.
   \section{\label{Sec.5}Positivity of CPD sequences}
Compared to Berger's characterization of
subnormal unilateral weighted shifts (see
Corollary~\ref{Ber-G-W}), the one given in
Theorem~\ref{cpdws} (see also
Theorem~\ref{wkwcpdws}) which characterizes CPD
unilateral weighted shifts requires answering
the question of when all terms of the sequence
appearing on the right-hand side of
\eqref{wnezero} are positive. In
Theorem~\ref{cpdws}(ii) there are three
parameters $b$, $c$ and $\nu$. Due to the number
of cases occurring in the proof of
Theorem~\ref{wkwcpdws} in which the parameter
$c$ is equal to $0$, treating the parameter $b$
as a variable seems to be the only possible
approach that may guarantee finding the solution
of this problem. That this is the case is shown
in Theorem~\ref{pos-wag-i}.

We begin by isolating hypothesis necessary to
state the main result of this section. Let $\nu$
be a compactly supported finite Borel measure on
$\rbb_+$ such that $\nu(\{1\})=0$, $c\in \rbb_+$
and $\{\gamma_n\}_{n=0}^{\infty}$ be the sequence
of continuous real functions on $\rbb$ defined by
   \begin{align} \label{wnezero-2}
\gamma_n(t) = 1 + t n + c n^2 + \int_{\rbb_+} Q_n
\D\nu, \quad t \in \rbb, \, n\in \zbb_+.
   \end{align}
Set
   \begin{align}  \notag
\vartheta & :=\sup\supp{\nu} \quad
\text{(convention: $\sup{\emptyset}=-\infty$)},
      \\ \notag
\varGamma_j & :=\int_{\rbb_+} \frac{1}{(1-x)^j}
\D \nu(x), \quad j=1,2 \text{ (whenever the
integrals make sense),}
   \\ \notag
\varOmega & := \{t\in \rbb \colon \gamma_n(t)
> 0  \text{ for all $n \in \zbb_+$}\},
   \\ \notag
\mathfrak{b} & :=\inf \varOmega.
   \end{align}
Clearly, $\varOmega = \{t\in \rbb \colon
\gamma_n(t) > 0 \text{ for all $n \in \nbb$}\}$.
Observing that the function $\gamma_n(\cdot)$ is
monotonically increasing for every $n \in
\zbb_+$, we get
   \begin{align} \label{zasl-dwa}
\textrm{if $t\in \varOmega$, then $[t,\infty)
\subseteq \varOmega$.}
   \end{align}
Since by \eqref{immed-pos}, $[0,\infty)
\subseteq \varOmega$ and $\lim_{t \to -\infty}
\gamma_n(t) = -\infty$ for every $n\in \nbb$, we
deduce that $\mathfrak{b} \in \rbb$. Combined
with \eqref{zasl-dwa}, this implies that the set
$\varOmega$ is an open or a closed subinterval
of $\rbb$ such that
   \begin{align*}
(\mathfrak{b},\infty) \subseteq \varOmega \subseteq
[\mathfrak{b},\infty).
   \end{align*}
Our goal is to determine whether or not
$\mathfrak{b}$ belongs to $\varOmega$. The answer
is given in Theorem~\ref{pos-wag-i} below. Before
doing it, we prove the following lemma.
   \begin{lem} \label{tak-tu-j}
Under the above assumptions, if $b_0 \in
[-\infty, \mathfrak{b})$ is such that
   \begin{align} \label{ju-ja-tr}
\lim_{n\to \infty} \gamma_n(t)=\infty, \quad t
\in (b_0,\infty),
   \end{align}
then $\mathfrak{b} \in (-\infty, 0)$ and
$\varOmega=(\mathfrak{b}, \infty)$; in
particular, $\gamma_n(\mathfrak{b}) \Ge 0$ for
all $n\in \zbb_+$ and
$\gamma_{n_0}(\mathfrak{b}) = 0$ for some
$n_0\in \nbb$.
   \end{lem}
   \begin{proof}
Since $0\in \varOmega$, it remains to show that
for every $b' \in \varOmega$, there exists
$\varepsilon \in (0,\infty)$ such that
$(b'-\varepsilon, \infty) \subseteq \varOmega$.
For, assume that $b' \in \varOmega$. Take $b_1
\in (b_0, b')$ and set
   \begin{align*}
J_{b_1} = \{n\in \zbb_+\colon \gamma_n(b_1) > 0\}.
   \end{align*}
Applying \eqref{ju-ja-tr} to $t=b_1$, we see that the
set $\zbb_+ \setminus J_{b_1}$ is finite. This and
   \begin{align*}
\lim_{t \to b'} \gamma_n(t) = \gamma_n(b') > 0, \quad
n\in \zbb_+,
   \end{align*}
imply that there exists $\eta \in (0,\infty)$ such
that
   \begin{align} \label{asd}
\gamma_n(t) > 0, \quad n\in \zbb_+ \setminus J_{b_1},
\, t\in (b'-\eta, b'].
   \end{align}
Since the affine function $\gamma_n$ is monotonically
increasing for every $n \in \zbb_+$, we get
   \begin{align} \label{asd-2}
0 < \gamma_n(b_1) \Le \gamma_n(t), \quad n \in
J_{b_1}, \, t\in [b_1,\infty).
   \end{align}
Combining \eqref{zasl-dwa}, \eqref{asd} and
\eqref{asd-2}, we get $(b'-\varepsilon, \infty)
\subseteq \varOmega$ for some $\varepsilon \in
(0,\infty)$. This completes the proof.
   \end{proof}
   Now we are ready to state and prove the main result
of this section.
   \begin{thm} \label{pos-wag-i}
Assume that $\{\gamma_n\}_{n=0}^{\infty}$ is as in
\eqref{wnezero-2}. Then
   \begin{enumerate}
   \item[(i)]
$\varOmega = (\mathfrak{b},\infty)$ with $-\infty <
\mathfrak{b} < 0$ if any of the following conditions
holds{\em :}
   \begin{enumerate}
   \item[(i-a)] $\vartheta > 1$,
   \item[(i-b)] $\vartheta \Le 1$ and $c>0$,
   \item[(i-c)] $\vartheta \Le 1$, $c=0$ and $\varGamma_1 = \infty$,
   \end{enumerate}
   \item[(ii)] if $\vartheta \Le 1$, $c=0$ and
$\varGamma_1 < \infty$, then
   \begin{enumerate}
   \item[(ii-a)] $\varOmega = (\mathfrak{b},\infty)$ with
$-\varGamma_1 < \mathfrak{b} < 0$ if $1 < \varGamma_2
\Le \infty$,
   \item[(ii-b)] $\varOmega = (-\varGamma_1,\infty)$
if\/\footnote{Obviously, $\nu=\delta_0$ implies that
$\varGamma_1=\varGamma_2=1$ and $\vartheta=0$.}
$\varGamma_2 = 1$ and $\nu=\delta_0$,
   \item[(ii-c)] $\varOmega = [-\varGamma_1,\infty)$
if $\varGamma_2 = 1$ and $\nu\neq \delta_0$,
   \item[(ii-d)] $\varOmega = [-\varGamma_1,\infty)$
if $\varGamma_2 < 1$.
   \end{enumerate}
   \end{enumerate}
   \end{thm}
   \begin{proof}
(i) We claim that if any of the cases \mbox{(i-a)},
\mbox{(i-b)} and \mbox{(i-c)} holds, then
\eqref{ju-ja-tr} is valid for $b_0=-\infty$. Indeed,
if $\vartheta
> 1$, then using \eqref{cul-dra} and
\eqref{str-2-zer}, we get
   \begin{align*}
\lim_{n\to\infty} \gamma_n(t) = \lim_{n\to\infty}
\int_{\rbb_+} Q_n \D\nu
\bigg(\frac{1+tn+cn^2}{\int_{\rbb_+} Q_n \D\nu} +
1\bigg) = \infty, \quad t\in \rbb.
   \end{align*}
If $\vartheta \Le 1$ and $c>0$, then by
Lemma~\ref{huhu-zima} we have
   \begin{align*}
\lim_{n\to\infty} \gamma_n(t) = \lim_{n\to\infty} n^2
\bigg(\frac{1+tn+cn^2}{n^2} + \int_{\rbb_+}
\frac{Q_n}{n^2} \D\nu \bigg) = \infty, \quad t\in
\rbb.
   \end{align*}
In turn, if $\vartheta \Le 1$, $c=0$ and $\varGamma_1
= \infty$, then \eqref{Eun-J} implies that
   \begin{align*}
\lim_{n\to\infty} \gamma_n(t) = \lim_{n\to\infty} n
\int_{\rbb_+} \frac{Q_n}{n} \D\nu \bigg(\frac{1+tn}{n
\int_{\rbb_+} \frac{Q_n}{n} \D\nu} + 1 \bigg) =
\infty, \quad t\in \rbb.
   \end{align*}
This proves our claim. Applying
Lemma~\ref{tak-tu-j}, we get (i).

(ii) Assume now that $\vartheta \Le 1$, $c=0$ and
$\varGamma_1 < \infty$. Using \eqref{Eun-J}, we obtain
   \begin{align} \label{zasl-dwu}
\lim_{n\to\infty} \gamma_n(t) = \lim_{n\to\infty} n
\bigg(\frac{1}{n} + t + \int_{\rbb_+} \frac{Q_{n}}{n}
\D\nu\bigg) = \begin{cases} -\infty & \text{if } t \in
(-\infty, -\varGamma_1),
   \\[1ex]
+ \infty & \text{if } t \in (-\varGamma_1,
\infty).
   \end{cases}
   \end{align}
We infer from \eqref{zasl-dwu} that
   \begin{gather} \label{ur-a-a}
- \varGamma_1 \Le \mathfrak{b}.
   \end{gather}
By the continuity of $\gamma_n$, we have
   \begin{align} \label{snig-tow}
\gamma_n(\mathfrak{b}) \Ge 0, \quad n\in \zbb_+.
   \end{align}
It follows from \eqref{bi-pas} that (recall that
$\gamma_0\equiv 1$)
   \begin{align} \label{dzw-dzu}
\gamma_n(-\varGamma_1) = 1 - \int_{[0,1)}
\frac{1-x^n}{(1-x)^2} \D \nu(x), \quad n \in \zbb_+.
   \end{align}
Applying Lebesgue's monotone convergence theorem to
\eqref{dzw-dzu}, we see that
   \begin{align} \label{dzw-dzs}
\gamma_n(-\varGamma_1) \searrow 1-\varGamma_2 \text{
as $n\to\infty$}.
   \end{align}

\mbox{(ii-a)} Suppose that $\varGamma_2 \in (1,
\infty]$. Then by \eqref{dzw-dzs}, $\lim_{n\to
\infty} \gamma_n(-\varGamma_1) \in [-\infty,
0)$. Combined with \eqref{ur-a-a} and
\eqref{snig-tow}, this implies that
$-\varGamma_1 < \mathfrak{b}$. Hence using
\eqref{zasl-dwu} and applying
Lemma~\ref{tak-tu-j} to $b_0=-\varGamma_1$, we
get \mbox{(ii-a)}.

\mbox{(ii-b)} If $\nu = \delta_0$, then the strict
monotonicity of $\gamma_n$ for $n\in \nbb$ implies
that
   \begin{align*}
\gamma_n(t) > \gamma_n(-1) = \gamma_n(-\varGamma_1)
\overset{\eqref{dzw-dzu}}= 0, \quad n\in \nbb, \, t\in
(-1, \infty).
   \end{align*}
Therefore, by \eqref{zasl-dwa}, $\varOmega =
(-1,\infty)$.

\mbox{(ii-c)} Assume now that $\varGamma_2 = 1$. We
claim that
   \begin{align} \label{ifroad}
\text{$\gamma_n(-\varGamma_1) > 0$ for all $n \in
\nbb$ if and only if $\nu \neq \delta_0$.}
   \end{align}
Indeed, if $n\in \nbb$ is fixed, then
   \begin{align*}
\gamma_n(-\varGamma_1) &
\overset{\eqref{dzw-dzu}}= \int_{[0,1)}
\frac{1}{(1-x)^2} \D \nu(x) - \int_{[0,1)}
\frac{1-x^n}{(1-x)^2} \D \nu(x)
   \\
&\hspace{1ex}= \int_{[0,1)} \frac{x^n}{(1-x)^2}
\D \nu(x) \Ge 0,
   \end{align*}
which implies that $\gamma_n(-\varGamma_1)=0$ if
and only if $\nu=\delta_0$ (use the fact that
$\varGamma_2=1$). This proves our claim. Since
$\gamma_0(\cdot)=1$, it follows from
\eqref{ifroad} that if $\nu \neq \delta_0$, then
$-\varGamma_1 \in \varOmega$, which together
with \eqref{zasl-dwa} and \eqref{ur-a-a} yields
\mbox{(ii-c)}.

\mbox{(ii-d)} Suppose that $\varGamma_2 < 1$. It
follows from \eqref{dzw-dzs} that $-\varGamma_1 \in
\varOmega$, which together with \eqref{zasl-dwa} and
\eqref{ur-a-a} implies that $\varOmega =
[-\varGamma_1, \infty)$. This completes the proof.
   \end{proof}
The quantity $\mathfrak{b}$ that appears in
Theorem~\ref{pos-wag-i} can be expressed in terms
of a new sequence $\{\zeta_n\}$ depending only on
the parameters $c$ and $\nu$.
   \begin{pro}\label{pos-wugi}
Let $\{\gamma_n\}_{n=0}^{\infty}$ be as in
\eqref{wnezero-2}. Then $\mathfrak{b} = -\inf_{n\in
\nbb} \zeta_n$, where
   \begin{align*}
\zeta_n=\frac{1}{n} + c n + \int_{\rbb_+}
\frac{Q_n}{n} \D \nu, \quad n \in \nbb.
   \end{align*}
Moreover, under the notation of Theorem~{\em
\ref{pos-wag-i}}, the following statements are
valid{\em :}
   \begin{enumerate}
   \item[(i)] in the cases {\em \mbox{(i-a)}}, {\em
\mbox{(i-b)}}, {\em \mbox{(i-c)}} and {\em
\mbox{(ii-a)}}, $\mathfrak{b} = -\min_{n\in \nbb}
\zeta_n$,
   \item[(ii)] in the case {\em \mbox{(ii-b)}},
$\mathfrak{b} = -1 = - \zeta_n$ for all $n \in \nbb$,
   \item[(iii)] in the cases {\em \mbox{(ii-c)}}
and {\em \mbox{(ii-d)}}, the sequence
\mbox{$\big\{\zeta_n\big\}_{n=1}^{\infty}$} is
strictly decreasing.
   \end{enumerate}
   \end{pro}
   \begin{proof}
(i) It follows from Theorem~\ref{pos-wag-i} that in
each of the cases \mbox{(i-a)}, \mbox{(i-b)},
\mbox{(i-c)} and \mbox{(ii-a)},
$\varOmega=(\mathfrak{b}, \infty)$ with $-\infty <
\mathfrak{b} < 0$. By the continuity of $\gamma_n$, we
have
   \begin{align} \label{Eun-num}
\gamma_n(\mathfrak{b}) \Ge 0, \quad n \in \nbb,
   \end{align}
which implies that
   \begin{align} \label{chrup}
\mathfrak{b} \Ge - \inf_{n\in \nbb} \zeta_n.
   \end{align}
Since $\mathfrak{b} \notin \varOmega$, we infer from
\eqref{Eun-num} that there exists $n_0 \in \nbb$ such
that $\gamma_{n_0}(\mathfrak{b}) = 0$. This yields
$\mathfrak{b} = - \zeta_{n_0}$. Combined with
\eqref{chrup}, this implies that $\mathfrak{b} =
-\min_{n\in \nbb} \zeta_n$.

(ii) Since $c=0$, this statement is easily seen
to be true (see
Theorem~\ref{pos-wag-i}\mbox{(ii-b)}).

(iii) By Theorem~\ref{pos-wag-i}, in the cases
\mbox{(ii-c)} and \mbox{(ii-d)},
$\gamma_n(-\varGamma_1) > 0$ for all $n\in \zbb_+$.
This together with \eqref{dzw-dzs} implies that the
sequence
$\big\{\frac{\gamma_n(-\varGamma_1)}{n}\big\}_{n=1}^{\infty}$
is strictly decreasing to $0$. Since
   \begin{align*}
\zeta_n = \frac{\gamma_n(-\varGamma_1)}{n} +
\varGamma_1, \quad n \in \nbb,
   \end{align*}
the sequence $\big\{\zeta_n\big\}_{n=1}^{\infty}$ is
strictly decreasing to $\varGamma_1$, which by
Theorem~\ref{pos-wag-i} is equal to $-\mathfrak{b}$.
Hence $\mathfrak{b} = -\inf_{n\in \nbb} \zeta_n$.

Summarizing the above considerations, we conclude that
$\mathfrak{b} = -\inf_{n\in \nbb} \zeta_n$ in all
cases. This completes the proof.
   \end{proof}
We conclude this section by making the following
observation related to Theorem~\ref{pos-wag-i}
and Proposition~\ref{pos-wugi}.
   \begin{rem}
It follows from Theorem~\ref{pos-wag-i} that
$-\infty < \mathfrak{b} < 0$ in all cases except
$c=0$ and $\nu=0$.
   \hfill $\diamondsuit$
   \end{rem}
   \section{The representing triplet of $\wlam$}
In this section we provide an explicit
description of the representing triplet
$(B,C,F)$ of a CPD unilateral weighted shift
$\wlam$. Recall that an operator $T\in
\ogr{\hh}$ is said to be {\em diagonal} with
diagonal terms $\{\xi_n\}_{n=0}^{\infty}
\subseteq \cbb$ with respect to an orthonormal
basis $\{f_n\}_{n=0}^{\infty}$ of $\hh$ if $Tf_n
= \xi_n f_n$ for every $n\in \zbb_+$. It is well
known that a diagonal operator $T$ with diagonal
terms $\{\xi_n\}_{n=0}^{\infty}$ is compact if
and only if $\lim_{n\to \infty} \xi_n = 0$.
   \begin{thm} \label{truplyt}
Let $\wlam$ be a CPD unilateral weighted shift
with weights
$\lambdab=\{\lambda_k\}_{k=0}^{\infty}$ and let
$(B,C,F)$ be the representing triplet of
$\wlam$. Then $B$, $C$ and $F(\varDelta)$, where
$\varDelta \in \borel{\rbb_+}$, are diagonal
operators with respect to
$\{e_k\}_{k=0}^{\infty}$ with diagonal terms
$\{b_k\}_{k=0}^{\infty}$,
$\{c_k\}_{k=0}^{\infty}$ and
$\{\nu_k(\varDelta)\}_{k=0}^{\infty}$,
respectively, given by
   \begin{align} \label{tyl-tru}
b_k=\frac{\hat\lambda_{k+1} - \hat\lambda_{k}
-c}{\hat\lambda_k}, \quad c_k =
\frac{c}{\hat\lambda_k}, \quad \nu_k(\varDelta)
= \frac{1}{\hat\lambda_k} \int_{\varDelta} x^k
\D \nu(x),
   \end{align}
where $\hat\lambda_k$ are as in \eqref{mur-hupy}
and $(b,c,\nu)$ is as in Theorem~{\em
\ref{cpdws}(ii)}. Moreover, the operator $C$ is
compact.
   \end{thm}
   \begin{proof}
All diagonal operators appearing in this proof
are regarded with respect to the standard
orthonormal basis $\{e_k\}_{k=0}^{\infty}$ of
$\ell^2$. First observe that for $n\in \zbb_+$,
$\wlam^{*n}\wlam^n$ is the diagonal operator
with diagonal terms
$\big\{\frac{\hat\lambda_{k+n}}{\hat\lambda_{k}}\big\}_{k=0}^{\infty}$.
Applying Theorem~\ref{cpdws} and
Lemma~\ref{szuf-eru} to $\gammab=\hat \lambdab$,
we deduce that
   \begin{align} \label{wije-su}
\frac{\hat\lambda_{k+n}}{\hat\lambda_k} = 1 +
b_k n + c_k n^2 + \int_{\rbb_+} Q_n(x) \D
\nu_k(x), \quad k,n\in \zbb_+.
   \end{align}
It follows from the proof of
Theorem~\ref{pos-wag-i}(i) that if any of the
cases \mbox{(i-a)} and \mbox{(i-b)} holds, then
$\lim_{k\to \infty} \hat\lambda_k=\infty$. This
implies that whatever $c$ is, $\lim_{k\to
\infty} c_k = 0$. Hence, the diagonal operator
with diagonal terms $\{c_k\}_{k=0}^{\infty}$,
say $\tilde C$, is a compact positive operator.
Since the sequence
$\{\lambda_k\}_{k=0}^{\infty}$ is bounded and
   \begin{align}  \label{pug-ra}
\frac{\hat\lambda_{k+n}}{\hat\lambda_k}
\overset{\eqref{self-map}}= \lambda_k^2 \cdots
\lambda_{k+n-1}^2, \quad k\in \zbb_+, \,n\in
\nbb,
   \end{align}
we deduce that for every $n\in \zbb_+$, the
sequence
$\big\{\frac{\hat\lambda_{k+n}}{\hat\lambda_k}\big\}_{k=0}^{\infty}$
is bounded. This together with \eqref{tyl-tru}
implies that the sequence
$\{b_k\}_{k=0}^{\infty}$ is bounded. Therefore,
the diagonal operator with diagonal terms
$\{b_k\}_{k=0}^{\infty}$, say $\tilde B$, is
bounded and selfadjoint. As a consequence of
\eqref{wije-su} and \eqref{pug-ra}, we conclude
that for every $n\in \zbb_+$, the sequence
$\{\int_{\rbb_+} Q_n(x) \D
\nu_k(x)\}_{k=0}^{\infty}$ is bounded. Since
$Q_2\equiv 1$, we see that for every $\varDelta
\in \borel{\rbb_+}$, the sequence
$\{\nu_k(\varDelta)\}_{k=0}^{\infty}$ is
bounded. Thus for every $\varDelta \in
\borel{\rbb_+}$, the diagonal operator with
diagonal terms
$\{\nu_k(\varDelta)\}_{k=0}^{\infty}$, say
$\tilde F(\varDelta)$, is bounded and positive.
Note that
   \begin{align} \label{hra-pi}
\is{\tilde F(\varDelta) f}{f} =
\sum_{k=0}^{\infty} |\is{f}{e_k}|^2
\nu_k(\varDelta), \quad \varDelta \in
\borel{\rbb_+}, \, f \in \ell^2.
   \end{align}
This together with \eqref{tyl-tru} implies that
$\tilde F$ is a compactly supported semispectral
measure such that $\tilde F(\{1\})=0$. Hence,
using standard measure and integration
techniques, we obtain
   \begin{align} \notag
\Big\langle \Big(\int_{\rbb_+} Q_n \D \tilde F
\Big) f,f \Big\rangle & = \int_{\rbb_+} Q_n(x)
\is{\tilde F(\D x) f}{f}
   \\  \label{alu-e}
& \hspace{-1ex}\overset{\eqref{hra-pi}} =
\sum_{k=0}^{\infty} |\is{f}{e_k}|^2
\int_{\rbb_+} Q_n(x) \D \nu_k(x), \quad f \in
\ell^2, \, n \in \zbb_+.
   \end{align}
This in turn yields
   \begin{align*}
\is{\wlam^{*n}\wlam^n f}{f} & =
\sum_{k=0}^{\infty} |\is{f}{e_k}|^2
\frac{\hat\lambda_{k+n}}{\hat\lambda_{k}}
   \\
& \hspace{-3.7ex}\overset{\eqref{wije-su} \&
\eqref{alu-e}} = \Big\langle \Big(I + \tilde B n
+ \tilde C n^2 + \int_{\rbb_+} Q_n \D \tilde
F\Big)f, f\Big\rangle, \quad f \in \ell^2, \, n
\in \zbb_+.
   \end{align*}
Applying the uniqueness part of
Theorem~\ref{cpdops}, we see that $(B, C, F) =
(\tilde B, \tilde C,\tilde F)$. This completes
the proof.
   \end{proof}
Recalling that the operator $C$ appearing in
Theorem~\ref{truplyt} is compact, the natural
question arises as to whether $B$ and
$F(\rbb_+)$ are compact operators. In general,
the answer is in the negative. Let us discuss
this in more detail. We begin with the
affirmative case.
   \begin{pro}\label{bnocr}
Under the assumptions and notation of
Theorem~{\em \ref{truplyt}}, suppose that
$\wlam$ satisfies Case~$4$ of the proof of
Theorem~{\em \ref{wkwcpdws}}. Then $B$ is
compact and $\wlam$ is not subnormal.
   \end{pro}
   \begin{proof}
That $B$ is compact follows immediately from
\eqref{tyl-tru} and \eqref{Hu-hy} applied to
$\gammab=\hat\lambdab$. To see that $\wlam$ is
not subnormal, observe that by
\cite[Theorem~3.4.1]{Ja-Ju-St20}, there is no
loss of generality in assuming that
$\frac{1}{x-1} \in L^1(F)$. Then
   \begin{align*}
\is{Be_0}{e_0} = b \neq \int_{\rbb_+}
\frac{1}{x-1} \D\nu(x) =
\Big\langle\Big(\int_{\rbb_+} \frac{1}{x-1} F(\D
x)\Big) e_0, e_0\Big\rangle,
   \end{align*}
which contradicts
\cite[Theorem~3.4.1(ii-b)]{Ja-Ju-St20}.
   \end{proof}
   The case when $B$ is not compact is
considered below.
   \begin{pro} \label{Bdziadnc}
Under the assumptions and notation of
Theorem~{\em \ref{truplyt}}, suppose that
$\vartheta:=\sup\supp{\nu} > 1$,
$\nu((1,\theta))=0$ for some $\theta \in
(1,\vartheta]$ and
   \begin{align} \label{sssttt}
\int_{[0,1)} \frac{1}{(x-1)^2} \D \nu(x)
<\infty.
   \end{align}
Then $B$ is not compact.
   \end{pro}
   \begin{proof}
By the compactness of $C$ and \eqref{tyl-tru},
it suffices to show that
   \begin{align}  \label{lstemi}
\liminf_{k\to \infty} \frac{\hat\lambda_{k+1} -
\hat\lambda_{k}}{\hat\lambda_k} > 0.
   \end{align}
Using the Cauchy-Schwarz inequality, we infer
from \eqref{sssttt} that $\int_{[0,1)}
\frac{1}{1-x} \D \nu(x) <\infty$. Hence, we can
define the functions $f,g\colon \nbb \to \rbb$
by
   \begin{align*}
f(k)&= b+ c(2k+1) - \int_{\rbb_+} \frac{1}{x-1}
\D \nu(x) + \int_{[0,1)} \frac{x^k}{x-1}\D
\nu(x), \quad k\in \nbb,
   \\
g(k)&= 1+ bk + c k^2 - \int_{\rbb_+}
\frac{1+k(x-1)}{(x-1)^2} \D \nu(x) +
\int_{[0,1)} \frac{x^k}{(x-1)^2}\D \nu(x), \quad
k\in \nbb.
   \end{align*}
It is easily seen that there is a positive real
number $\alpha$ such that
   \begin{align} \label{hhhu1}
\text{$|f(k)| \Le \alpha k$ and $|g(k)| \Le
\alpha k^2$ for all $k\in \nbb$.}
   \end{align}
Using the assumption that $\nu((1, \theta)) =
0$, we deduce that
   \begin{align*}
\int_{(1,\infty)} \frac{x^k}{(x-1)^2} \D \nu(x)
\Ge \theta^k \int_{[\theta,\vartheta]}
\frac{1}{(x-1)^2} \D \nu(x), \quad k \in \nbb.
   \end{align*}
Since $\nu([\theta, \vartheta])$ is positive, so
is $\int_{[\theta,\vartheta]} \frac{1}{(x-1)^2}
\D \nu(x)$. Combined with \eqref{hhhu1}, this
yields
   \begin{align} \label{ssddd}
\text{$\lim_{k\to\infty}
\frac{f(k)}{\int_{(1,\infty)}
\frac{x^k}{(x-1)^2}\D \nu(x)}=0$ \; and \;
$\lim_{k\to\infty} \frac{g(k)}{\int_{(1,\infty)}
\frac{x^k}{(x-1)^2}\D \nu(x)}=0$.}
   \end{align}
Observe also that
   \begin{align} \label{hxcvw}
\frac{\int_{(1,\infty)} \frac{x^k}{x-1} \D
\nu(x)}{\int_{(1,\infty)} \frac{x^k}{(x-1)^2}\D
\nu(x)} \Ge \theta - 1, \quad k\in \nbb.
   \end{align}
It follows from \eqref{rnx-1}, \eqref{del1} and
\eqref{wnezero} that
   \allowdisplaybreaks
   \begin{align}  \notag
\frac{\hat\lambda_{k+1} -
\hat\lambda_{k}}{\hat\lambda_k} & = \frac{b +
c(2k+1) + \int_{\rbb_+} \frac{x^k-1}{x-1}\D
\nu(x)}{1 + bk + c k^2 + \int_{\rbb_+}
\frac{x^k-1-k(x-1)}{(x-1)^2} \D\nu(x)}
   \\  \notag
& = \frac{f(k) + \int_{(1,\infty)}
\frac{x^k}{x-1} \D \nu(x)}{g(k) +
\int_{(1,\infty)} \frac{x^k}{(x-1)^2} \D\nu(x)}
   \\  \label{kdnb}
& = \frac{\frac{f(k)}{\int_{(1,\infty)}
\frac{x^k}{(x-1)^2}\D \nu(x)} +
\frac{\int_{(1,\infty)} \frac{x^k}{x-1} \D
\nu(x)}{\int_{(1,\infty)} \frac{x^k}{(x-1)^2}\D
\nu(x)}}{\frac{g(k)}{\int_{(1,\infty)}
\frac{x^k}{(x-1)^2}\D \nu(x)} + 1}, \quad k\in
\nbb.
   \end{align}
Finally, we deduce \eqref{lstemi} from
\eqref{ssddd}, \eqref{hxcvw} and \eqref{kdnb}.
This completes the proof.
   \end{proof}
Now we discuss the question of compactness of
$F(\rbb_+)$.
   \begin{pro}\label{cupuc}
Under the assumptions and notation of
Theorem~{\em \ref{truplyt}}, suppose that
$\vartheta:=\sup\supp{\nu} > 1$ and $1 \notin
\supp{\nu}$. Then $F(\rbb_+)$ is not compact.
   \end{pro}
   \begin{proof}
For brevity, we define the real functions $r$
and $\varLambda$ on $\nbb$ by
   \begin{align*}
r(k) & =1+bk + c k^2 - \int_{\rbb_+} \frac{1 +k
(x-1)}{(x-1)^2}\D \nu(x), \quad k \in \nbb,
   \\
\varLambda(k) & =\int_{\rbb_+}
\frac{x^k}{(x-1)^2}\D \nu(x), \quad k\in \nbb.
   \end{align*}
Set $\theta = \frac{1+\vartheta}{2}$ and observe
that $\int_{[\theta,\vartheta]}
\frac{1}{(x-1)^2} \D\nu(x) > 0$. Since
   \begin{align*}
\varLambda(k) \Ge \theta^k
\int_{[\theta,\vartheta]} \frac{1}{(x-1)^2}
\D\nu(x), \quad k \in \nbb,
   \end{align*}
we see that
   \begin{align} \label{troj-band}
\lim_{k\to \infty} \frac{r(k)}{\varLambda(k)} =
0.
   \end{align}
It follows from \eqref{rnx-1}, \eqref{wnezero}
and \eqref{tyl-tru} that
   \begin{align*}
\nu_k(\rbb_+) & = \frac{\int_{\rbb_+} x^k \D
\nu(x)}{1+bk + c k^2 + \int_{\rbb_+} \frac{x^k-1
- k (x-1)}{(x-1)^2}\D \nu(x)}
   \\
& = \frac{\int_{\rbb_+} x^k \D \nu(x)}{r(k) +
\varLambda(k)} = \frac{\frac{\int_{\rbb_+} x^k
\D
\nu(x)}{\varLambda(k)}}{\frac{r(k)}{\varLambda(k)}
+ 1}
   \\
& \Ge
\frac{\mathrm{dist}(1,\supp{\nu})^2}{\frac{r(k)}{\varLambda(k)}
+ 1}, \quad k \in \nbb.
   \end{align*}
Combined with \eqref{troj-band}, this implies
that
   \begin{align*}
\liminf_{k\to \infty} \nu_k(\rbb_+) \Ge
\mathrm{dist}(1,\supp{\nu})^2 > 0.
   \end{align*}
By Theorem~\ref{truplyt}, $F(\rbb_+)$ is not
compact.
   \end{proof}
We now give an explicit example covering all
cases discussed above.
   \begin{exa} \label{oliun}
Fix $\theta \in (0,\infty)\backslash \{1\}$ and
define the sequence $\gammab =
\{\gamma_n\}_{n=0}^{\infty}$ of real numbers by
   \begin{align*}
\gamma_n= 1 + Q_n(\theta), \quad n \in \zbb_+.
   \end{align*}
Clearly, $\gamma_n > 0$ for all $n\in \zbb_+$
(see \eqref{klaud} and \eqref{immed-pos}),
$\gammab$ is a CPD sequence of exponential
growth and its representing triplet $(b,c,\nu)$
is given by $b=c=0$ and $\nu=\delta_{\theta}$.
By Theorem~ \ref{wkwcpdws}, the unilateral
weighted shift $\wlam$ with weights
$\lambdab=\{\lambda_n\}_{n=0}^{\infty}$ defined
by \eqref{wagua} is CPD and
$\hat\lambdab=\gammab$. Consider now two cases.

{\sc Case 1.} $\theta < 1$.

Then by Proposition~\ref{bnocr}, $B$ is compact.
Since by \eqref{rnx-1} and \eqref{tyl-tru},
   \begin{align*}
\nu_k(\rbb_+) = \frac{\theta^k}{1+Q_k(\theta)}
\to 0 \quad \text{as } k \to \infty,
   \end{align*}
we infer from Theorem~\ref{truplyt} that
$F(\rbb_+)$ is compact.

{\sc Case 2.} $\theta > 1$.

Then by Proposition~\ref{Bdziadnc}, $B$ is not
compact. In turn, by Proposition~\ref{cupuc},
$F(\rbb_+)$ is not compact.
   \hfill $\diamondsuit$
   \end{exa}
   \section{CPD backward extensions}
In this section, we solve the $1$-step backward
extension problem for CPD unilateral weight\-ed
shifts, which consists of finding a necessary
and sufficient condition for the existence of a
positive real number $t$ for which the
unilateral weighted shift $W_{(t,\lambda_0,
\lambda_1, \ldots)}$ is CPD under the assumption
that $W_{(\lambda_0, \lambda_1, \ldots)}$ is. We
generalize this procedure to the CPD $n$-step
backward extension using recurrence formulas.

We begin with the crucial lemma related to CPD
sequences.
   \begin{lem} \label{bakre}
Suppose that $\gammab =
\{\gamma_n\}_{n=0}^{\infty}$ is a CPD sequence
of exponential growth and $\supp{\nu} \subseteq
\rbb_+$, where $(b,c,\nu)$ is the representing
triplet of $\gammab$. Let $\theta \in \rbb$. Set
$\gamma_{-1}=\theta$. Then the following
conditions are~equivalent\/\footnote{In the
integral formula in (ii), we adhere to the
convention that $\frac{1}{0}=\infty$. It will be
used in the subsequent parts of this paper.}{\em
:}
   \begin{enumerate}
   \item[(i)] $\betab:=\{\gamma_{n-1}\}_{n=0}^{\infty}$
is a CPD sequence of exponential growth and
$\supp{\nu_{\betab}} \subseteq \rbb_+$, where
$(b_{\betab}, c_{\betab}, \nu_{\betab})$ is the
representing triplet of $\betab$,
   \item[(ii)] $\int_{\rbb_+} \frac{1}{x} \D
\nu(x) \Le \theta + \gamma_1 - 2(\gamma_0 + c)$.
   \end{enumerate}
Moreover, if {\em (i)} holds, then
$\nu(\{0\})=0$ and
   \begin{align} \label{brus1}
b_{\betab} & = \gamma_0 - \theta - c,
   \\ \label{brus2}
c_{\betab}& =c,
   \\ \label{brus3}
\nu_{\betab}(\varDelta) & = \int_{\varDelta}
\frac{1}{x} \D\nu(x) + \nu_{\betab}(\{0\})
\delta_0(\varDelta), \quad \varDelta \in
\borel{\rbb_+},
   \end{align}
with
   \begin{align} \label{brus4}
\nu_{\betab}(\{0\}) = \theta + \gamma_1
-2(\gamma_0 + c) - \int_{\rbb_+} \frac{1}{x}
\D\nu(x).
   \end{align}
   \end{lem}
   \begin{proof}
(i)$\Rightarrow$(ii) Since $\beta_n=
\gamma_{n-1}$ for $n\in \zbb_+$, we get $\gammab
= \betab^{(1)}$ (see \eqref{bekon} for the
definition). Clearly, by \eqref{euju} and Lemma~
\ref{szuf-eru} applied to $\betab$ in place of
$\gammab$ with $k=1$, we deduce that
\eqref{brus1} and \eqref{brus2} hold and
   \begin{align} \label{kriwan}
\nu(\varDelta) = \int_{\varDelta} x \D
\nu_{\betab}(x), \quad \varDelta \in
\borel{\rbb_+}.
   \end{align}
This implies that $\nu(\{0\})=0$, and
consequently
   \begin{align*}
\nu_{\betab}(\varDelta) &
\overset{\eqref{kriwan}}= \int_{\varDelta
\backslash \{0\}} \frac{1}{x} \D\nu(x) +
\nu_{\betab}(\{0\}) \delta_0 (\varDelta), \quad
\varDelta \in \borel{\rbb_+},
   \end{align*}
which implies \eqref{brus3}. We now compute
$\nu_{\betab}(\{0\})$. It follows from
\eqref{cdr4}, \eqref{brus1}, \eqref{brus2} and
the identity $Q_2\equiv 1$ that
   \begin{align*}
\gamma_1 = \beta_2 = \beta_0 + 2b_{\betab} + 4
c_{\betab} + \nu_{\betab}(\rbb_+) = \theta + 2
(\gamma_0 - \theta - c) + 4 c +
\nu_{\betab}(\rbb_+).
   \end{align*}
This yields
   \begin{align*}
\nu_{\betab}(\rbb_+) = \theta + \gamma_1 - 2
(\gamma_0 + c).
   \end{align*}
Substituting $\varDelta=\rbb_+$ into
\eqref{brus3} and using $\nu(\{0\})=0$, we
conclude that (ii) and \eqref{brus4} hold. This
also proves the ``moreover'' part.

(ii)$\Rightarrow$(i) Set $\tilde b = \gamma_0 -
\theta - c$, $\tilde c = c$ and
   \begin{align*}
\tilde \nu(\varDelta) & = \int_{\varDelta}
\frac{1}{x} \D\nu(x) + \Big(\theta + \gamma_1
-2(\gamma_0 + c) - \int_{\rbb_+} \frac{1}{x}
\D\nu(x)\Big) \delta_0(\varDelta), \quad
\varDelta \in \borel{\rbb_+}.
   \end{align*}
In particular, we have
   \begin{align} \label{nutylpa}
\tilde \nu(\rbb_+) & = \theta + \gamma_1
-2(\gamma_0 + c), \quad \varDelta \in
\borel{\rbb_+}.
   \end{align}
Clearly, $\tilde b \in \rbb$, $\tilde c \in
\rbb_+$ and by (ii), $\tilde \nu$ is a compactly
supported finite Borel measure on $\rbb_+$ such
that $\tilde \nu(\{1\})=0$. Define the sequence
$\tilde \betab=\{\tilde
\beta_n\}_{n=0}^{\infty}$ by
   \begin{align} \label{tidufor}
\tilde \beta_n = \theta + \tilde bn + \tilde c
n^2 + \int_{\rbb_+} Q_n(x) \D\tilde \nu(x),
\quad n\in \zbb_+.
   \end{align}
By Theorem~\ref{cpd-expon}, $\tilde \betab$ is a
CPD sequence of exponential growth. To complete
the proof, it is enough to show that
   \begin{align} \label{urwa}
\tilde \betab^{(1)} = \gammab.
   \end{align}
For, applying Lemma~\ref{szuf-eru} to $\tilde
\betab$ in place of $\gammab$ with $k=1$, we
deduce that $\tilde \betab^{(1)}$ is a CPD
sequence of exponential growth and
   \begin{align} \notag
\tilde b_1 = (\tilde \beta_2- \tilde \beta_1) -c
& \overset{\eqref{tidufor}}= \big(\tilde b + 3 c
+ \tilde \nu (\rbb_+)\big) - c
   \\ \notag
& \overset{\eqref{nutylpa}}= (\gamma_0 - \theta
- c) + 3c + (\theta + \gamma_1 -2(\gamma_0 + c))
- c
   \\ \label{comychcu}
& \hspace{1ex} = \gamma_1-\gamma_0 - c
\overset{\eqref{euju}}= b.
   \end{align}
Again, using Lemma~\ref{szuf-eru} in the same
context leads to
   \begin{align} \label{comychcr}
\tilde c_1 = \tilde c = c,
   \end{align}
and
   \begin{align} \label{comychcx}
\tilde \nu_1(\varDelta) = \int_{\varDelta} x \,
\D \tilde \nu(x) = \nu(\varDelta), \quad
\varDelta \in \borel{\rbb_+}.
   \end{align}
Finally, we have
   \begin{align*}
\tilde \beta^{(1)}_0 = \tilde \beta_1
\overset{\eqref{tidufor}} = \theta + \tilde b +
\tilde c = \gamma_0.
   \end{align*}
Combined with \eqref{comychcu}, \eqref{comychcr}
and \eqref{comychcx}, this implies \eqref{urwa},
which completes the proof.
   \end{proof}
   We now solve the $1$-step backward extension
problem for CPD unilateral weight\-ed shifts.
   \begin{thm} \label{bireks}
Suppose that $\wlam$ is a CPD unilateral
weighted shift with weights
$\lambdab=\{\lambda_n\}_{n=0}^{\infty}$ and
$(b,c,\nu)$ is its scalar representing triplet.
Let $t \in (0,\infty)$. Then the following
conditions are~equivalent{\em :}
   \begin{enumerate}
   \item[(i)] $W_{(t,\lambda_0, \lambda_1,
\ldots)}$ is CPD,
   \item[(ii)] $\frac{1}{t^2} \Ge \int_{\rbb_+}
\frac{1}{x} \D \nu(x) + 1 + c - b$.
   \end{enumerate}
Moreover, if $W_{(t,\lambda_0, \lambda_1,
\ldots)}$ is CPD and $(b_t,c_t,\nu_t)$ is its
scalar representing triplet, then $\nu(\{0\})=0$
and
   \begin{align} \label{trut1}
b_t & = t^2(1 - c) - 1,
   \\ \label{trut2}
c_t & =t^2 c,
   \\ \label{trut3}
\nu_t(\varDelta) & = t^2 \int_{\varDelta}
\frac{1}{x} \D\nu(x) + \nu_t(\{0\})
\delta_0(\varDelta), \quad \varDelta \in
\borel{\rbb_+},
   \end{align}
with
   \begin{align} \label{trut4}
\nu_t(\{0\}) = 1 - t^2\Big(\int_{\rbb_+}
\frac{1}{x} \D\nu(x) + 1 + c - b\Big).
   \end{align}
   \end{thm}
   \begin{proof}
Assume that (i) holds. Note that $t^2
\hat\lambdab = (t^2, t^2\lambda_0^2, t^2
\lambda_0^2 \lambda_1^2, \ldots)$ is a CPD
sequence of exponential growth with the
representing triplet $(t^2 b, t^2 c, t^2 \nu)$.
It follows from Theorem~\ref{cpdws} applied to
$W_{(t,\lambda_0, \lambda_1, \ldots)}$ that the
sequence $(1, t^2, t^2\lambda_0^2, t^2
\lambda_0^2 \lambda_1^2, \ldots)$ is CPD and
$\supp{\nu_t} \subseteq \rbb_+$. Applying
Lemma~\ref{bakre} to $\gammab = (t^2,
t^2\lambda_0^2, t^2 \lambda_0^2 \lambda_1^2,
\ldots)$ and $\theta = 1$, we conclude that (ii)
and \eqref{trut1}-\eqref{trut4} are valid.
Reversing the above reasoning shows that (ii)
implies (i), which completes the proof.
   \end{proof}
Given $n \in \nbb$ and a unilateral weighted
shift $\wlam$ with weights
$\lambdab=\{\lambda_k\}_{k=0}^{\infty}$, we say
that $\wlam$ has a {\em CPD $n$-step backward
extension} if there exists a sequence
$\{t_j\}_{j=1}^n$ of positive real numbers such
that the unilateral weighted shift
$W_{(t_n,\ldots,t_1, \lambda_0, \lambda_1,
\ldots)}$ is CPD. We call $W_{(t_n,\ldots,t_1,
\lambda_0, \lambda_1, \ldots)}$ a {\em CPD
$n$-step backward extension} of $\wlam$. In case
$\wlam$ has a CPD $k$-step backward extension
for every $k\in \nbb$, we say that $\wlam$ has a
{\em CPD $\infty$-step backward extension}.
Since the restriction of a CPD operator to its
closed invariant subspace is CPD and
$W_{(\lambda_1,\lambda_2, \ldots)}$ is unitarily
equivalent to $\wlam|_{\overline{\ob{\wlam}}}$,
we deuce that $W_{(\lambda_1,\lambda_2,
\ldots)}$ is CPD whenever $\wlam$ is. By
induction, we~have
   \begin{align} \label{wtwpi}
   \begin{minipage}{72ex}
{\em if $W_{(t_{n}, \ldots, t_1, \lambda_0,
\lambda_1, \ldots)}$ is CPD, then $\wlam$ is CPD
and $W_{(t_{k}, \ldots, t_1, \lambda_0,
\lambda_1, \ldots)}$ is a CPD $k$-step backward
extension of $\wlam$ for $1 \Le k\Le n$.}
   \end{minipage}
   \end{align}

The solution of the $n$-step backward extension
problem for CPD unilateral weighted shifts takes
the following form.
   \begin{thm} \label{nthbuk}
Suppose that $\wlam$ is a CPD unilateral
weighted shift with weights
$\lambdab=\{\lambda_k\}_{k=0}^{\infty}$ and
$(b,c,\nu)$ is its scalar representing triplet.
Let $n\Ge 2$ be an integer and let
$\{t_j\}_{j=1}^n\subseteq (0,\infty)$. Set
$t_0:=\lambda_0= (1 + b + c)^{1/2}$ and
$\boldsymbol t_k := (t_k,\ldots,t_1)$ for $k =
1, \ldots, n$. Then the following statements are
equivalent{\em :}
   \begin{enumerate}
   \item[(i)] $W_{(t_n,\ldots,t_1,\lambda_0, \lambda_1,\ldots)}$
is a CPD $n$-step backward extension of $\wlam$,
   \item[(ii)] the sequence $\{t_j\}_{j=1}^n$
satisfies the following
conditions\/\footnote{with the convention that
$\prod_{j=1}^{0} t_j^2=1$}{\em :}
   \begin{align*}
\frac{1}{t_k^2} & = \Big(\prod_{j=1}^{k-1} t_j^2
\Big) \Big(\int_{\rbb_+} \frac{1}{x^{k}} \D
\nu(x) + 2 c\Big) + 2 - t_{k-1}^2, \quad k=1,
\ldots, n-1,
   \\
\frac{1}{t_n^2} & \Ge \Big(\prod_{j=1}^{n-1}
t_j^2\Big) \Big(\int_{\rbb_+} \frac{1}{x^{n}} \D
\nu(x) + 2 c\Big) + 2 - t_{n-1}^2.
   \end{align*}
   \end{enumerate}
Moreover, if {\em (ii)} holds, then for $k=1,
\ldots, n$, $W_{(t_k,\ldots,t_1, \lambda_0,
\lambda_1, \ldots)}$ is a CPD $k$-step backward
extension of $\wlam$ with the scalar
representing triplet $(b_{\boldsymbol t_k},
c_{\boldsymbol t_k}, \nu_{\boldsymbol t_k})$
given~by{\em :}
   \begin{align*}
b_{\boldsymbol t_k} & = t_k^2 - c \prod_{j=1}^k
t_j^2 - 1,
   \\
c_{\boldsymbol t_k} & = c \prod_{j=1}^k t_j^2,
   \\
\nu_{\boldsymbol t_k}(\varDelta) & =
\Big(\prod_{j=1}^k t_j^2\Big) \int_{\varDelta}
\frac{1}{x^k} \D \nu(x) + \nu_{\boldsymbol
t_k}(\{0\}) \delta_0(\varDelta), \quad \varDelta
\in \borel{\rbb_+},
   \\
\nu_{\boldsymbol t_k}(\{0\}) & = 1 -
\Big(\prod_{j=1}^{k} t_j^2\Big)
\Big(\int_{\rbb_+} \frac{1}{x^{k}} \D \nu(x) + 2
c\Big) + t_k^2(t_{k-1}^2-2).
   \end{align*}
In particular, $\nu_{\boldsymbol t_k}(\{0\})=0$
for $k=1, \ldots, n-1$.
   \end{thm}
   \begin{proof}
This can be proved by applying induction on $n$
by using \eqref{wtwpi} and Theorem~\ref{bireks}.
We leave the details to the reader.
   \end{proof}
   \begin{rem} \label{pours}
As for the integrals $\int_{\rbb_+}
\frac{1}{x^k} \D \nu(x)$ that appear in
Theorem~\ref{nthbuk}, it is worth noting that
   \begin{align*}
\textit{if $\int_{\rbb_+} \frac{1}{x^{k+1}} \D
\nu(x) < \infty$ for some $k \in \nbb$, then
$\int_{\rbb_+} \frac{1}{x^{j}} \D \nu(x) <
\infty$ for $j=1, \ldots, k$.}
   \end{align*}
This is because $\nu$ is finite and compactly
supported.
   \hfill $\diamondsuit$
   \end{rem}
We now give a handy criterion for the existence
of a CPD $n$-step backward extension. It will be
used in Example~\ref{przyktwofor}.
   \begin{pro}\label{wlr1}
Suppose that $p\in \nbb$. Let $\wlam$ be a CPD
unilateral weighted shift with weights $\lambdab
= \{\lambda_k\}_{k=0}^{\infty}$ and $(b,c,\nu)$
be its scalar representing triplet such~that
   \begin{align} \label{wireles}
\int_{\rbb_+} \frac{1}{x^k} \D \nu(x) < \infty,
\quad k= 1, \ldots, p.
   \end{align}
Let $\{\sigma_k\}_{k=1}^{\infty}$ be the
sequence defined by the following
formal\/\footnote{Note that $\sigma_j$ may
vanish or be infinite for some $j$.} recurrence
relations{\em :}
   \begin{align} \label{urure}
\sigma_1 & = \int_{\rbb_+} \frac{1}{x} \D \nu(x)
+ 1 + c - b,
   \\ \label{ururl}
\sigma_{k+1} & = \frac{1}{\prod_{j=1}^{k}
\sigma_j} \Big(\int_{\rbb_+} \frac{1}{x^{k+1}}
\D \nu(x) + 2 c\Big) + 2 - \frac{1}{\sigma_k},
\quad k\in \nbb.
   \end{align}
If $\sigma_k > 0$ for $k=1, \ldots, p-1$
$($which can be dropped if $p= 1$\/$)$ and
$\sigma_p \Ge 1$, then the following statements
hold with\footnote{Observe that $n_{\lambdab}$
can be equal to $+\infty$ (see
Example~\ref{przyktwofor}).}
$n_{\lambdab}=\sup\big\{k \in \nbb \colon
\int_{\rbb_+} \frac{1}{x^k} \D \nu(x) <
\infty\big\}${\em :}
   \begin{enumerate}
   \item[(i)]
$\sigma_k \in (0,\infty)$ for $k=1, \ldots,
p-1$,
   \item[(ii)]
$\sigma_k \in [1,\infty)$ for every integer $k$
such that $p \Le k \Le n_{\lambdab}$,
   \item[(iii)] $\wlam$ has a CPD $n_{\lambdab}$-step backward
extension.
   \end{enumerate}
   \end{pro}
   \begin{proof}
Set $J=\{j \in \nbb\colon 1 \Le j \Le
n_{\lambdab}\}$. It follows from \eqref{wireles}
and Remark~\ref{pours} that $n_{\lambdab}\Ge p$
and
   \begin{align} \label{justu}
J = \Big\{j\in \nbb\colon \int_{\rbb_+}
\frac{1}{x^j} \D \nu(x) < \infty\Big\}.
   \end{align}
Assume that $\sigma_j > 0$ for $j=1, \ldots,
p-1$ and $\sigma_p \Ge 1$. Using induction and
\eqref{wireles}-\eqref{ururl}, one can verify
that
   \begin{align} \label{leje}
\sigma_j \in (0,\infty), \quad j=1, \ldots, p,
   \end{align}
which justifies (i). The proof of (ii) is by
induction on $k$. Suppose that for some
unspecified integer $k$ such that $p \Le k <
n_{\lambdab}$, $\sigma_j \in [1,\infty)$ for
$j=p, \ldots, k$. This together with
\eqref{leje} and the induction hypothesis
implies that $\sigma_j \in (0,\infty)$ for $j=1,
\ldots, k$. Since $k+1 \in J$, we deduce from
\eqref{ururl}, \eqref{justu} and $\sigma_k \Ge
1$ that $\sigma_{k+1} \in \rbb$ and
   \begin{align*}
\sigma_{k+1} - 1 & = \frac{1}{\prod_{j=1}^{k}
\sigma_j} \Big(\int_{\rbb_+} \frac{1}{x^{k+1}}
\D \nu(x) + 2 c\Big) + \Big(1 -
\frac{1}{\sigma_k}\Big)
   \\
& \Ge \frac{1}{\prod_{j=1}^{k} \sigma_j}
\Big(\int_{\rbb_+} \frac{1}{x^{k+1}} \D \nu(x) +
2 c\Big) \Ge 0,
   \end{align*}
which shows that $\sigma_{k+1}\in [1,\infty)$.
This completes the induction argument. Thus (ii)
holds. Finally (iii) is a direct consequence of
(i),(ii) and Theorem~\ref{nthbuk}.
   \end{proof}
Applying Proposition~\ref{wlr1} to $p=1$, we get
the following.
   \begin{cor} \label{sipris}
Let $\wlam$ be a CPD unilateral weighted shift
with weights $\lambdab =
\{\lambda_k\}_{k=0}^{\infty}$ and $(b,c,\nu)$ be
its scalar representing triplet such that $b\Le
c$. Then $\wlam$ has a CPD $\infty$-step
backward extension if and only if $\int_{\rbb_+}
\frac{1}{x^k} \D \nu(x) < \infty$ for all $k\in
\nbb$.
   \end{cor}
As shown in Example~\ref{przyktwofor}, the
inequality $b\Le c$ is not necessary for a
unilateral weighted shift to have CPD
$\infty$-step backward extension.

Recall that in the case of the $1$-step backward
extension problem for subnormal unilateral
weighted shifts, the parameter $t \in
(0,\infty)$ for which the unilateral weighted
shift $W_{(t,\lambda_0, \lambda_1,\ldots)}$ is
subnormal (provided it exists) has a finite
upper bound depending on $(\lambda_0,
\lambda_1,\ldots)$ (see
\cite[Proposition~8]{Cur90}). In the
corresponding problem for completely
hyperexpansive unilateral weighted shifts, the
parameter $t$ has a positive lower bound (see
\cite[Corollary~3.3]{J-J-S06}). In contrast to
these two cases, when considering CPD unilateral
weighted shifts $W_{(t,\lambda_0,
\lambda_1,\ldots)}$, it may happen that there is
neither a finite upper nor a positive lower
bound for the parameter $t$. A similar effect
appears in the CPD $n$-step backward extension
problem.
   \begin{exa} \label{muritru}
Fix $k\in \nbb$ and take any $\theta \in
\big(\frac{1}{k}, \frac{1}{k-1} \big)$ (with the
convention that $\frac{1}{0}=\infty$). Define
the sequence
$\gammab=\{\gamma_n\}_{n=0}^{\infty}$ by
   \begin{align*}
\gamma_n = 1 + n \theta, \quad n\in \zbb_+.
   \end{align*}
Clearly, $\gammab$ is a CPD sequence of
exponential growth satisfying \eqref{gl-odd1}.
Its representing triplet $(b,c,\nu)$ is given by
$b=\theta$, $c=0$ and $\nu=0$. Applying
Theorem~\ref{wkwcpdws}, we get a CPD unilateral
weighted shift $\wlam$ with weights
$\lambdab=\{\lambda_n\}_{n=0}^{\infty}$ such
that $\hat\lambdab = \gammab$. The weights of
$\wlam$ are given by
   \begin{align*}
\lambda_n =
\sqrt{\frac{1+(n+1)\theta}{1+n\theta}}, \quad
n\in \zbb_+.
   \end{align*}
In fact, $\wlam$ is a $2$-isometry (see
\cite[Lemma~6.1(ii)]{Ja-St}). With notation as
in Proposition~\ref{wlr1}, we verify that
$\{\sigma_n\}_{n=0}^{\infty} \subseteq \rbb
\setminus \{0\}$ and
   \begin{align*}
\sigma_n = \frac{1-n\theta}{1-(n-1)\theta},
\quad n \in \nbb.
   \end{align*}
This implies that $\sigma_n > 0$ for $n=1,
\ldots, k-1$ and $\sigma_k < 0$. Hence, by
Theorem~\ref{nthbuk}, there exists a finite
sequence $\{t_j\}_{j=1}^{k-1} \subseteq
(0,\infty)$ (which can be dropped if $k=1$) such
that the unilateral weighted shift $W_{(t_k,
\ldots, t_1, \lambda_0, \lambda_1, \ldots)}$ is
CPD for every $t_k \in (0,\infty)$, and $\wlam$
has no CPD \mbox{$(k+1)$}-step backward
extension.
   \hfill $\diamondsuit$
   \end{exa}
   \section{Flatness of CPD unilateral weighted shifts}
In \cite[Theorem~6]{Sta66}, Stampfi proved that if two
consecutive weights of a subnormal unilateral weighted shift
$\wlam$ are equal, then $\wlam$ is flat, meaning that the
sequence of its weights stabilizes starting from the second
weight. In this section, we show that this is no longer true
for CPD unilateral weighted shifts, namely, the number of
consecutive equal weights has to be increased to four and
this number is optimal (under some constrains discussed in
detail below).

We begin with the case where four consecutive
weights not containing the initial one are
equal, and none of them is equal to one.
   \begin{thm} \label{4weights-1}
Suppose that $\wlam$ is a CPD unilateral
weighted shift with weights
$\lambdab=\{\lambda_n\}_{n=0}^{\infty}$ and
$\kappa\in \nbb$ is such that
   \begin{align} \label{oruusa}
\lambda_\kappa = \lambda_{\kappa+1} =
\lambda_{\kappa+2} = \lambda_{\kappa+3} \neq 1.
   \end{align}
Then $\lambda_0 \Le \lambda_1 = \lambda_n$ for
all $n\in \nbb$ and $\wlam$ is subnormal with
the Berger measure $\mu$ given by
   \begin{align} \label{mumier}
\mu=(1-Z) \delta_0 + Z \delta_{\lambda_1^2}
\quad \text{with} \quad Z =
\Big(\frac{\lambda_0}{\lambda_1}\Big)^2.
   \end{align}
   \end{thm}
   \begin{proof}
Set $\theta = \lambda_{\kappa}^2$,
$\gamma_n=\hat\lambda_n$ for $n\in \zbb_+$ (see
\eqref{mur-hupy}) and $\gammab =
\{\gamma_n\}_{n=0}^{\infty}$. By \eqref{oruusa},
we have (see \eqref{bekon} for notation)
   \begin{align} \label{takieconieco}
\gamma^{(\kappa)}_{n} = \gamma_{\kappa}
\theta^n, \quad n=0, 1, 2, 3, 4.
   \end{align}
Straightforward computations show that
   \begin{align}  \notag
(\triangle \gammab^{(\kappa)})_n & =
\gamma_{\kappa} (\theta-1) \theta^n, \quad
n=0,1,2,3,
   \\ \label{trokj}
(\triangle^2 \gammab^{(\kappa)})_n & =
\gamma_{\kappa} (\theta-1)^2 \theta^n, \quad
n=0,1,2.
   \end{align}
Let $(b,c,\nu)$ be the scalar representing
triplet of $\wlam$. It follows from
\eqref{wnezero} and \eqref{del2} that
   \begin{align*}
(\triangle^2 \gammab)_n = \int_{\rbb_+} x^n (\nu
+ 2 c \delta_1) (\D x), \quad n\in \zbb_+.
   \end{align*}
This, in turn, implies that
   \begin{align}  \label{trookj}
\alpha_n:=(\triangle^2 \gammab)^{(\kappa)}_n =
\int_{\rbb_+} x^n \D \rho_{\kappa}(x), \quad
n\in \zbb_+,
   \end{align}
where $\rho_{\kappa}$ is a finite compactly
supported Borel measure on $\rbb_+$ given by
   \begin{align} \label{huru}
\rho_{\kappa}(\varDelta) = \int_{\varDelta}
x^{\kappa} (\nu + 2 c \delta_1)(\D x), \quad
\varDelta \in \borel{\rbb_+}.
   \end{align}
According to \eqref{fateu}, \eqref{trokj} and
\eqref{trookj}, we have
   \begin{align} \label{ibjub}
\alpha_n=\int_{\rbb_+} x^n \D
\rho_{\kappa}(x)=\gamma_{\kappa} (\theta-1)^2
\theta^n, \quad n=0,1,2.
   \end{align}
Since $\theta \neq 1$, we infer from
\eqref{trookj}, \eqref{ibjub} and \eqref{foot-1}
that the Stieltjes moment sequence
$\{\alpha_n\}_{n=0}^{\infty}$ is non-degenerate.
It follows from \eqref{ibjub} that
   \begin{align*}
\alpha_1^2 = \alpha_0 \alpha_2.
   \end{align*}
Applying Lemma~\ref{ajjaj} with $k=0$ , we
deduce that there exist $\xi, \zeta\in
(0,\infty)$ such that $\rho_{\kappa} = \xi
\delta_{\zeta}$. As a consequence of
\eqref{ibjub}, we obtain
   \begin{align} \label{dirdu}
\gamma_{\kappa} (\theta-1)^2 \theta^n =
\xi\zeta^n, \quad n=0,1,2.
   \end{align}
Substituting $n=0$, we get $\xi=\gamma_{\kappa}
(\theta-1)^2$, which together with \eqref{dirdu}
gives $\zeta=\theta$. Hence $\rho_{\kappa} = \xi
\delta_{\theta}$, which yields
   \begin{align*}
\xi \delta_{\theta} (\varDelta)
\overset{\eqref{huru}}= \int_{\varDelta}
x^{\kappa} (\nu + 2 c \delta_1)(\D x), \quad
\varDelta \in \borel{\rbb_+}.
   \end{align*}
Since $\theta \neq 1$, this implies that $c=0$
and that there exist $u \in \rbb_+$ and $v \in
(0,\infty)$ such that
   \begin{align} \label{tolkobyta}
\nu = u \delta_0 + v \delta_{\theta}.
   \end{align}
Combined with \eqref{wnezero}, this shows that
   \begin{align} \notag
\gamma_n & = 1 + b n + u Q_n(0) + v Q_n(\theta)
   \\ \notag
& \hspace{-1ex}\overset{\eqref{rnx-1}}= 1 + b n
+ u (n-1) + v \frac{\theta^n-1 - n
(\theta-1)}{(\theta-1)^2}
   \\ \label{zrabi}
& = X + Y n + Z \theta^n, \quad n \in \nbb,
   \end{align}
where $X:=1 -u - Z$, $Y:= b + u -
\frac{v}{\theta-1}$ and
$Z:=\frac{v}{(\theta-1)^2}$. By \eqref{zrabi},
we have
   \begin{align*}
\gamma^{(\kappa)}_n = (X + \kappa Y) + Y n + (Z
\theta^{\kappa}) \theta^n, \quad n \in \zbb_+.
   \end{align*}
This together with \eqref{takieconieco} leads to
   \begin{align*}
\gamma_{\kappa} \theta^n = (X + \kappa Y) + Y n
+ (Z \theta^{\kappa}) \theta^n, \quad n=0, 1, 2,
3, 4.
   \end{align*}
Therefore, we get
   \begin{align} \label{alelipa}
(X + \kappa Y) + Y n = (\gamma_{\kappa} - Z
\theta^{\kappa}) \theta^n, \quad n=0, 1, 2, 3,
4.
   \end{align}
Substituting $n=0$ into \eqref{alelipa} gives
   \begin{align} \label{nidra}
X + \kappa Y = \gamma_{\kappa} - Z
\theta^{\kappa}.
   \end{align}
Next substituting $n=1,2$ into \eqref{alelipa}
and using \eqref{nidra}, we obtain
   \begin{align*}
(X + \kappa Y) + Y & = (X + \kappa Y) \theta,
   \\
(X + \kappa Y) + 2 Y & = (X + \kappa Y)
\theta^2.
   \end{align*}
Since $\theta \neq 1$, we conclude that the
above homogeneous system of linear equations has
only the solution $X=Y=0$. It follows from
\eqref{zrabi} that
   \begin{align} \label{zurv}
\gamma_n = Z \theta^n, \quad n \in \nbb.
   \end{align}
(Note that $Z > 0$ because $\gamma_1 > 0$.) This
implies that
   \begin{align} \label{zurv0}
\lambda_n \overset{\eqref{self-map}}=
\sqrt{\frac{\gamma_{n+1}} {\gamma_n}} =
\sqrt{\theta}, \quad n\in \nbb.
   \end{align}
Since $1 - u - Z = X = 0$ and $u\in \rbb_+$, we
infer from \eqref{zurv0} that
   \begin{align} \label{grenb}
\frac{\lambda_0^2}{\lambda_1^2} =
\frac{\gamma_1}{\theta} \overset{\eqref{zurv}}=
Z = 1 - u \Le 1.
   \end{align}
Using \eqref{zurv}, \eqref{zurv0} and
\eqref{grenb}, we verify that the measure $\mu$
given by \eqref{mumier} satisfies \eqref{Stieq},
so by Corollary~\ref{Ber-G-W}, $\wlam$ is
subnormal and $\mu$ is its Berger measure. This
completes the proof.
   \end{proof}
If the weights in the sequence \eqref{oruusa}
are equal to $1$, then we can reduce their
number to two ($\lambda_0$ is still excluded).
   \begin{thm} \label{2weights-1}
Suppose that $\wlam$ is a CPD unilateral
weighted shift with weights
$\lambdab=\{\lambda_n\}_{n=0}^{\infty}$ and
$\kappa\in \nbb$ is such that
   \begin{align*}
\lambda_\kappa = \lambda_{\kappa+1} = 1.
   \end{align*}
Then $\lambda_0 \Le 1 = \lambda_n$ for all $n\in
\nbb$ and $\wlam$ is subnormal with the Berger
measure $\mu$ given by
   \begin{align*}
\mu = (1-\lambda_0^2) \delta_0 + \lambda_0^2
\delta_1.
   \end{align*}
   \end{thm}
   \begin{proof}
We will modify the proof of
Theorem~\ref{4weights-1}. Set $\gamma_n = \hat
\lambda_n$ for $n\in \zbb_+$. Let $(b,c,\nu)$ be
the scalar representing triplet of $\wlam$.
First note that
   \begin{align} \label{kolynu}
\gamma^{(\kappa)}_n = \gamma_{\kappa}, \quad n=
0,1,2,
   \end{align}
and thus
   \begin{align*}
(\triangle^2 \gammab^{(\kappa)})_0 = 0.
   \end{align*}
This implies that
   \begin{align*}
\int_{\rbb_+} x^{\kappa} (\nu + 2 c \delta_1)(\D
x) \overset{\eqref{huru}}=\rho_{\kappa}(\rbb_+)
\overset{\eqref{trookj}} = (\triangle^2
\gammab)^{(\kappa)}_0 \overset{\eqref{fateu}} =
(\triangle^2 \gammab^{(\kappa)})_0 = 0.
   \end{align*}
Since $\kappa \Ge 1$, we deduce that $c=0$ and
$\supp{\nu} \subseteq \{0\}$. Hence, we have
   \begin{align} \label{dua2}
\gamma_n \overset{\eqref{wnezero}}= \big(1 -
\nu(\{0\})\big) + \widetilde Y n, \quad n \in
\nbb,
   \end{align}
where $\widetilde Y:=b+\nu(\{0\})$, and
consequently
   \begin{align*}
\gamma_{\kappa} \overset{\eqref{kolynu}}=
\gamma^{(\kappa)}_n \overset{\eqref{dua2}}=
\widetilde X + \widetilde Y n, \quad n=0,1,2,
   \end{align*}
where $\widetilde X = 1 - \nu(\{0\}) + \kappa
\widetilde Y$. This implies that $\widetilde
Y=0$ and thus by \eqref{dua2}, we get
   \begin{align*}
1 - \nu(\{0\}) = \gamma_n=\gamma_1= \lambda_0^2,
\quad n\in \nbb.
   \end{align*}
Therefore $\lambda_0 \Le 1$ and $\lambda_n =
\sqrt{\frac{\gamma_{n+1}}{\gamma_n}}=1$ for all
$n\in \nbb$. The remaining part of the proof
proceeds as in Theorem~\ref{4weights-1}.
   \end{proof}
We now consider the case of the first four equal
weights, none of which is equal to $1$.
   \begin{thm} \label{firwts}
Suppose that $\wlam$ is a CPD unilateral
weighted shift with weights
$\lambdab=\{\lambda_n\}_{n=0}^{\infty}$ such
that
   \begin{align*}
\lambda_0 = \lambda_1 = \lambda_2 = \lambda_3
\neq 1.
   \end{align*}
Then $\frac{1}{\lambda_0} \wlam$ is the
unilateral shift.
   \end{thm}
   \begin{proof}
We again modify the proof of
Theorem~\ref{4weights-1}. For the reader's
convenience, we will point out the most
important differences. First, in
\eqref{tolkobyta} we have $u=0$. Consequently,
the identity \eqref{zrabi} takes the form
   \begin{align*}
\gamma_n = X + Y n + Z \theta^n, \quad n \in
\zbb_+.
   \end{align*}
Further differences appear in the formulas
\eqref{zurv} and \eqref{zurv0}, so now we have
   \begin{align*}
\text{$\gamma_n = Z \theta^n$ and $\lambda_n =
\sqrt{\theta}$ for all $n\in \zbb_+$.}
   \end{align*}
This completes the proof.
   \end{proof}
Finally, arguing as in the proof of
Theorem~\ref{2weights-1}, we obtain the variant
of this theorem for $\kappa=0$.
   \begin{thm} \label{firetom}
Suppose that $\wlam$ is a CPD unilateral
weighted shift with weights
$\lambdab=\{\lambda_n\}_{n=0}^{\infty}$ such
that
   \begin{align*}
\lambda_0 = \lambda_1 = 1.
   \end{align*}
Then $\wlam$ is the unilateral shift.
   \end{thm}
We now show that the numbers of consecutive
equal weights appearing in
Theorems~\ref{4weights-1}, \ref{2weights-1},
\ref{firwts} and \ref{firetom} are optimal. The
first counterexample concerns the case of three
consecutive equal weights. It is noteworthy that
if two consecutive weights of a subnormal
unilateral weighted shift are equal, all
weights, except the first, are equal (see
\cite[Theorem~6]{Sta66}). Thus, for subnormal
unilateral weighted shifts the optimal number of
consecutive equal weights is exactly $2$.
   \begin{exa} \label{przyktwofor}
Fix $\theta \in (1,\infty)\backslash \{3\}$ and
define the triplet $(b,c,\nu)$ by $b= \theta -
1$, $c=0$ and
   \begin{align*}
\nu= \frac{1}{2} (\theta - 1)^2
\big(\delta_{\frac{1}{3}\theta} +
\delta_{\frac{5}{3}\theta}\big).
   \end{align*}
Let $\gammab=\{\gamma_n\}_{n=0}^{\infty}$ be the
CPD sequence given by \eqref{gl-odd2}. Clearly,
by \eqref{immed-pos}, $\gammab$ satisfies
\eqref{gl-odd1}. Applying
Theorem~\ref{wkwcpdws}, we get a CPD unilateral
weighted shift $\wlam$ with weights $\lambdab =
\{\lambda_n\}_{n=0}^{\infty}$ such that
$\hat\lambdab = \gammab$. Straightforward
computations yield
   \begin{align} \label{gimtre}
\gamma_n = \theta^n, \quad n=0,1,2,3,
   \end{align}
and
   \begin{align} \label{trata}
\gamma_4 = \frac{1}{9}\theta^2(13 \theta^2 - 8
\theta + 4).
   \end{align}
(Note that the expression on the right-hand side
of \eqref{trata} is equal to $\theta^4$ if and
only if $\theta =1$, which is not the case, so
$\gamma_4 \neq \theta^4$.) Hence, by
\eqref{wagua} and \eqref{gimtre} we have
   \begin{align} \label{prusy}
\lambda_0=\lambda_1=\lambda_2 = \sqrt{\theta}
\neq 1,
   \end{align}
and
   \begin{align} \label{prus}
\lambda_3 = \frac{\sqrt{13 \theta^2 - 8 \theta +
4}}{3\sqrt{\theta}} \neq \sqrt{\theta}.
   \end{align}
This shows that Theorem~\ref{firwts} is no
longer true if the number of consecutive equal
weights is decreased to $3$.

To cover the case $\kappa \Ge 1$ discussed in
Theorem~\ref{4weights-1}, we will use
Proposition~\ref{wlr1}. With notation as in this
proposition, note that $n_{\lambdab}=\infty$ and
   \begin{align} \label{silg}
\sigma_1 = g_1(\theta),
   \end{align}
where $g_1\colon (1,\infty) \to \rbb$ is defined
by
   \begin{align*}
g_1(x) := \frac{4x^2 - 8 x +9}{5x}, \quad x \in
(1,\infty).
   \end{align*}
Since
   \begin{align*}
g_1(x) - 1 = \frac{(x-1)(4x-9)}{5x}, \quad x \in
(1,\infty),
   \end{align*}
we see that
   \begin{align} \label{fgdw}
g_1(x)\Ge 1, \quad x \in
\Big[\frac{9}{4},\infty\Big).
   \end{align}
Fix any $\theta \in \big[\frac{9}{4},\infty\big)
\setminus \{3\}$. Using \eqref{silg} and
\eqref{fgdw} and applying Proposition~\ref{wlr1}
with $p=1$, we conclude that $\wlam$ has a CPD
$\infty$-step backward extension. Summarizing,
we have proved that for every $n \in \nbb$,
there exists a sequence $\{t_j\}_{j=1}^{n}
\subseteq (0,\infty)$ such that the unilateral
weighted shift $W_{\lambdab^{(n)}}$ with weights
   \begin{align*}
\lambdab^{(n)} = (t_{n}, \ldots, t_1,
\boxed{\lambda_0, \lambda_1, \lambda_2},
\lambda_3, \ldots),
   \end{align*}
is CPD. By \eqref{prusy}, \eqref{prus} and
Theorems~\ref{4weights-1} and \ref{firwts},
$\lambda_0=\lambda_1=\lambda_2 \neq 1$ and $t_1
\neq \lambda_0$. This means that regardless of
the value of $\kappa\Ge 1$,
Theorem~\ref{4weights-1} ceases to be true if
the number of consecutive equal weights is
decreased to $3$.

Applying Proposition~\ref{wlr1}, now to any
$p\Ge 1$, and using computer simulations, one
can confirm that the above conclusion is true
for $\theta \in (1,\infty)\backslash \{3\}$.
   \hfill $\diamondsuit$
   \end{exa}
The second counterexample shows that the number
of consecutive equal weights appearing in
Theorems~\ref{2weights-1} and \ref{firetom} is
optimal. It is only interesting for
non-subnormal CPD unilateral weighted shifts
because the class of subnormal operators, unlike
CPD ones, is scalable (see
\cite[Corollary~3.4.7]{Ja-Ju-St20}).
   \begin{exa} \label{gusv}
Let $\nu$ be a compactly supported finite Borel
measure on $\rbb_+$ such that $\nu(\{1\})=0$ and
let $c\in (0,\infty)$. Set $b=-c$. Let
$\gammab=\{\gamma_n\}_{n=0}^{\infty}$ be the CPD
sequence defined by \eqref{gl-odd2}. Because of
\eqref{immed-pos}, $\gammab$ satisfies
\eqref{gl-odd1}. Applying
Theorem~\ref{wkwcpdws}, we obtain a CPD
unilateral weighted shift $\wlam$ with weights
$\lambdab = \{\lambda_n\}_{n=0}^{\infty}$ such
that $\hat\lambdab = \gammab$. Since $c > 0$,
one can infer from \eqref{wagua} that
$\lambda_0=1$ and $\lambda_1 \neq 1$, which
means that $\wlam$ is not the unilateral shift.
In fact, $\wlam$ is not subnormal because $c> 0$
(use Theorem~\ref{truplyt} and \cite[Theorem
3.4.1]{Ja-Ju-St20}). This shows that
Theorem~\ref{firetom} is no longer true if the
first weight is equal to $1$.

Finally, assuming additionally that
$\int_{\rbb_+} \frac{1}{x^k} \D \nu(x) < \infty$
for all $k\in \nbb$ and applying
Corollary~\ref{sipris} to the above unilateral
weighted shift $\wlam$, we conclude that
regardless of the value of $\kappa\Ge 1$,
Theorem~\ref{2weights-1} ceases to be true if
$\lambda_{\kappa}=1$ (cf.\
Example~\ref{przyktwofor}).
   \hfill $\diamondsuit$
   \end{exa}
   \bibliographystyle{amsalpha}
   
   \end{document}